\documentclass[11pt]{article}

\usepackage{amsmath,amssymb,amsfonts,textcomp,amsthm,xifthen,psfrag,graphicx,color}
\usepackage[T1]{fontenc}
\usepackage{tikz}

\date{}

\usepackage{pgfplots}
\usepgfplotslibrary{colorbrewer}
\pgfplotsset{compat=newest}
\usepgfplotslibrary{external}
\newtheorem{theorem}{Theorem}

\newtheorem{cor}[theorem]{Corollary}
\newtheorem{prop}[theorem]{Proposition}
\newtheorem{remark}[theorem]{Remark}
\theoremstyle{definition} 

\newcommand{\ip}[2]{\langle#1\hspace*{.5mm},#2\rangle}
\newcommand{\dual}[2]{\langle#1\hspace*{.5mm},#2\rangle}
\newcommand{\vdual}[2]{(#1\hspace*{.5mm},#2)}

\newcommand{\ds}{\displaystyle}
\newcommand{\transp}{\mathsf{T}}

\newcommand{\diam}{\mathrm{diam}}

\def\div{\operatorname{div}}
\def\Div{\operatorname{\bf div}}
\def\dDiv{\operatorname{div\, \bf div}}

\def\grad{\nabla}
\def\Grad{\boldsymbol{\varepsilon}}
\def\Ggrad{\Grad\grad}

\newcommand{\traceO}[1]{\mathrm{tr}_{#1}^{\mathrm{grad}}}
\newcommand{\tracediv}[1]{\mathrm{tr}_{#1}^{\mathrm{div}}}
\newcommand{\traceDD}[1]{\mathrm{tr}_{#1}^{\mathrm{dDiv}}} 
\newcommand{\traceGG}[1]{\mathrm{tr}_{#1}^{\mathrm{Ggrad}}} 

\newcommand{\cPP}{d(\Omega)}
\newcommand{\cC}{\ensuremath{\mathcal{C}}}
\newcommand{\cCinv}{\ensuremath{\mathcal{C}^{-1}}}

\newcommand{\R}{\ensuremath{\mathbb{R}}}

\newcommand{\bH}{\ensuremath{\boldsymbol{H}}}
\newcommand{\HH}{\ensuremath{\mathbb{H}}}

\newcommand{\bL}{\ensuremath{\mathbf{L}}}
\newcommand{\LL}{\ensuremath{\mathbb{L}}}
\newcommand{\PP}{\ensuremath{\mathbb{P}}}

\def\MM{\mathbf{M}}
\def\QQ{\mathbf{Q}}
\def\dQQ{{\boldsymbol\delta\!\QQ}}

\newcommand{\g}{\ensuremath{\mathbf{g}}}
\newcommand{\bn}{\ensuremath{\mathbf{n}}}

\newcommand{\bv}{\ensuremath{\boldsymbol{v}}}
\newcommand{\bpsi}{\ensuremath{\boldsymbol{\psi}}}

\newcommand{\cT}{\ensuremath{\mathcal{T}}}

\newcommand{\HOtr}[1]{\ensuremath{H^{1/2}_{00}(#1)}}
\newcommand{\Hdivtr}[1]{\ensuremath{H^{-1/2}(#1)}}

\newcommand{\Hdiv}[1]{\ensuremath{\bH(\mathrm{div},#1)}}

\newcommand{\HdDiv}[1]{{\HH(\dDiv,#1)}}
\newcommand{\HdDivz}[1]{{\HH_0(\dDiv,#1)}}

\newcommand{\cS}{\ensuremath{\mathcal{S}}}

\newcommand{\bt}{\ensuremath{\mathbf{t}}}
\newcommand{\cP}{\ensuremath{\mathcal{P}}}
\newcommand{\OO}{\ensuremath{\mathcal{O}}}
\newcommand{\EE}{\ensuremath{\mathcal{E}}}

\newcommand{\GG}{\ensuremath{\mathbf{G}}}

\newcommand{\tu}{{\widehat u}}
\newcommand{\btu}{{\widehat{\boldsymbol{u}}}}
\newcommand{\tQ}{{\widehat \bq}}
\newcommand{\tM}{{\widehat{\boldsymbol{m}}}}

\newcommand{\dv}{{\delta\!v}}

\newcommand{\tv}{{\widehat v}}
\newcommand{\btv}{{\widehat{\boldsymbol{v}}}}

\newcommand{\tsigma}{{\widehat \sigma}}

\newcommand{\bsigma}{{\boldsymbol\sigma}}
\newcommand{\btau}{{\boldsymbol\tau}}
\newcommand{\dbtau}{{\boldsymbol\delta\!\btau}}
\newcommand{\ttau}{{\widehat\tau}}
\newcommand{\bq}{{\boldsymbol{q}}}
\newcommand{\bu}{\boldsymbol{u}}
\newcommand{\bg}{\boldsymbol{g}}
\newcommand{\bw}{\boldsymbol{w}}

\newcounter{constantsnumber}
\def\setc#1{
  \ifthenelse{\equal{#1}{poinc}}{C_{\rm edge}}{ 
   \refstepcounter{constantsnumber}
   \label{const#1}C_{\theconstantsnumber}}}

\def\c#1{
  \ifthenelse{\equal{#1}{poinc}}{C_{\rm edge}}{ 
    C_{\ref{const#1}}}}

\title{A robust DPG method for large domains
\thanks{Supported by CONICYT through FONDECYT projects 11170050, 1190009.}}
\author{
Thomas~F\"uhrer\thanks{
Facultad de Matem\'aticas, Pontificia Universidad Cat\'olica de Chile,
Avenida Vicu\~na Mackenna 4860, Santiago, Chile,
email: {\tt \{tofuhrer,nheuer\}@mat.uc.cl}}
\and
Norbert Heuer$^\dagger$
}

\begin{document}
\maketitle

\begin{abstract}
We observe a dramatic lack of robustness of the DPG method when solving problems on large domains
and where stability is based on a Poincar\'e-type inequality. We show how robustness
can be re-established by using appropriately scaled test norms.
As model cases we study the Poisson problem and the Kirchhoff--Love plate bending model,
and also include fully discrete variants where optimal test functions are approximated.
Numerical experiments for both model problems,
including an-isotropic domains and mixed boundary conditions, confirm our findings.


\end{abstract}

\section{Introduction}

The discontinuous Petrov--Galerkin method with optimal test functions (DPG method)
aims at automatically satisfying the discrete inf-sup condition. In the setting proposed
by Demkowicz and Gopalakrishnan, this is achieved by selecting optimal test functions
defined in product test spaces (``broken spaces'') with inner products,
cf.~\cite{DemkowiczG_11_CDP}. Apart from its inherent discrete stability,
this method has several advantages like generating symmetric, positive definite system matrices
(for real unknowns) and providing built-in error estimators,
cf.~\cite{DemkowiczGN_12_CDP,CarstensenDG_14_PEC}.
Perhaps less known, but frequently observed in numerical experiments,
is that built-in error estimators reliably generate
appropriate mesh refinements, even when starting with very coarse meshes of a few elements.
To cite from the conclusions of \cite{DemkowiczGN_12_CDP}:
the ``adaptive process \ldots starts with few elements only'' and
``one does not need to start with a mesh that reflects an expertise on the problem.''
This is also reflected by the numerical experiments in~\cite{DemkowiczH_13_RDM},
it is a feature that is extremely relevant when solving singularly perturbed problems.

It turns out that most DPG experiments reported so far are for problems on relatively small
domains, even when dealing with practically relevant models, see, e.g., the
paper \cite{BramwellDGQ_12_Lhp} on linear elasticity with focus on locking.
In our research on plate bending models we observed a locking phenomenon
that appears when solving more realistic cases, e.g., the bending of larger plates.
(For a systematic discussion of locking phenomena we refer to \cite{BabuskaS_92_LRF}.)
This locking can be traced back to using Poincar\'e-type estimates
in the stability analysis of adjoint problems.
They fail to be uniformly stable for large domains.
The robustness of the DPG method is equivalent to the uniform stability of the
corresponding adjoint problem, when using ultraweak variational formulations
where only derivatives of test functions appear.
Now, such formulations are the preferred ones for DPG approximations,
for practical reasons (direct representation of field variables, simple basis functions,
straightforward implementation of symmetry of tensors, simplicity of conformity)
and theoretical motivation (standard analytical tools).
There are only few exceptions, e.g.~\cite{DemkowiczG_13_PDM,BroersenS_15_PGD}.
Therefore, the fact that there is a domain-inflicted locking phenomenon of the DPG method
for ultraweak formulations on large domains is a major threat for practical applications.
Of course, this domain locking will only appear for problems whose adjoint stability
is based upon such a Poincar\'e-type estimate, including linear elasticity where
Korn's second inequality is used to prove ellipticity. The possible lack of robustness
depends on the geometry and the type of boundary condition.
For large domains, the standard DPG scheme exhibits a pre-asymptotic range even for the
simple Poisson problem (without reaction!).
This becomes a major problem when solving higher-order problems like plate bending models.
In those cases, locking may occur even for moderately sized domains,
as our numerical experiments illustrate.

In this paper we show how the DPG scheme with ultraweak formulation can be fixed
to become robust with respect to the size of the domain. One has to weight
individual terms of the test norm appropriately, and ensure that analytical estimates
hold uniformly with respect to involved parameters. We show this for the Poisson
problem, cf.~\cite{DemkowiczG_11_ADM}, and a scaled model of Kirchhoff--Love plate bending
as analyzed in \cite{FuehrerHN_19_UFK,FuehrerH_19_FDD}. We also indicate
how to proceed in the fully discrete cases of approximated optimal test functions,
coined ``practical DPG method'' by Gopalakrishnan and Qiu in \cite{GopalakrishnanQ_14_APD}.
Indeed, we show that our lowest-order scaled scheme is domain robust in the discrete setting.

An overview of the remainder is as follows.
In Section~\ref{sec_P} we consider the Poisson problem in detail. We first recall
the standard setting and its quasi-optimal convergence, discussing in Remark~\ref{rem_P}
the relevance of the Poincar\'e--Friedrichs inequality. In \S\ref{sec_Pd}
we present the setting that leads to a domain-robust DPG scheme, and prove this
fact (Theorem~\ref{thm_Pd}). Afterwards we analyze the fully discrete case.
Section~\ref{sec_KL} presents and analyzes our domain-robust DPG scheme for
the Kirchhoff--Love plate bending model. Both clamped and simply supported plates are
considered. After presenting the model problem, in \S\ref{sec_KL_traces}
we introduce scaled norms and analyze trace operators in these norms. Our domain-robust
DPG scheme is then defined and analyzed in \S\ref{sec_KL_DPG}.
The fully discrete setting is subject of \S\ref{sec_KLh}.
Numerical experiments for both model problems are reported in Section~\ref{sec_num}.
There, we consider different boundary conditions and sequences of isotropic and
an-isotropic domains. We end with some conclusions.

In our bounds below, the notation $a\lesssim b$ means that there is a constant $C>0$,
independent of the domain $\Omega$, the underlying mesh $\cT$, the scaling parameter $d$
to be introduced, and involved functions, such that $a\le C\,b$.

\section{Poisson problem} \label{sec_P}

We consider the Poisson problem with and without reaction,
\begin{align} \label{Poisson}
  u\in H^1_0(\Omega):\qquad
  &-\Delta u + \gamma u = f \quad \text{ in } \Omega.
\end{align}
Here, $\Omega\subset\R^n$, $n\in\{2,3\}$, is a bounded, polygonal/polyhedral Lipschitz domain.
Furthermore, $f\in L_2(\Omega)$ is a given function and $\gamma\ge 0$.

\subsection{Standard DPG method}

Let us introduce some spaces. For $\omega\subset\Omega$,
$L_2(\omega)$, $H^1(\omega)$, and $\Hdiv{\omega}$ denote standard Sobolev spaces.
The $L_2(\omega)$-bilinear form and norm are $\vdual{\cdot}{\cdot}_\omega$ and $\|\cdot\|_\omega$.
Furthermore, $H^1_0(\Omega)\subset H^1(\Omega)$
is the subspace of $H^1$-functions with vanishing trace on $\partial\Omega$.
We use the norms
\begin{align*}
   \|v\|_{1,\omega}^2 &:= \|v\|_\omega^2 + \|\nabla v\|_\omega^2\quad (v\in H^1(\omega)),\quad
   \|\btau\|_{\div,\omega}^2 := \|\btau\|_\omega^2 + \|\div\btau\|_\omega^2\quad (\btau\in\Hdiv{\omega})
\end{align*}
and drop the index $\omega$ when $\omega=\Omega$.
We consider a (family of) mesh(es) $\cT$ of open Lipschitz elements $T$ on $\Omega$ with boundaries
$\partial T$, so that any two elements do not intersect and $\bar\Omega=\cup_{T\in\cT}\bar T$.
Induced product spaces are
\[
   H^1(\cT):=\Pi_{T\in\cT} H^1(T),\qquad \Hdiv{\cT}:=\Pi_{T\in\cT} \Hdiv{T}
\]
with norms $\|\cdot\|_{1,\cT}$ and $\|\cdot\|_{\div,\cT}$, respectively.

In the following $\vdual{\cdot}{\cdot}_\cT$ denotes the $\cT$-piecewise $L_2$-bilinear form,
that is, appearing differential operators will be applied only locally on elements $T\in\cT$.
Using this notation, there are two canonical (local) trace operators,
$\traceO{T}:\;H^1(T)\to (\Hdiv{T})'$ and $\tracediv{T}:\;\Hdiv{T}\to (H^1(T))'$
with support on $\partial T$ ($T\in\cT$), defined by
\begin{align}
   \label{dual_gradT}
   &\dual{\traceO{T}(v)}{\dbtau}_{\partial T} :=
   \vdual{v}{\div\dbtau}_T + \vdual{\grad v}{\dbtau}_T
   &&(v\in H^1(T),\ \dbtau\in\Hdiv{T}),\\
   \label{dual_divT}
   &\dual{\tracediv{T}(\btau)}{\dv}_{\partial T} :=
   \vdual{\btau}{\grad \dv}_T + \vdual{\div\btau}{\dv}_T
   &&(\btau\in\Hdiv{T},\ \dv\in H^1(T)),
\end{align}
and the product variants $\traceO{}:\;H^1(\cT)\to (\Hdiv{\cT})'$,
$\tracediv{}:\;\Hdiv{\cT}\to (H^1(\cT))'$ defined by
\begin{align}
   \label{dual_grad}
   &\dual{\traceO{}(v)}{\dbtau}_\cS := \sum_{T\in\cT} \dual{\traceO{T}(v)}{\dbtau}_{\partial T}
   &&(v\in H^1(\cT),\ \dbtau\in\Hdiv{\cT}),\\
   \label{dual_div}
   &\dual{\tracediv{}(\btau)}{\dv}_\cS := \sum_{T\in\cT} \dual{\tracediv{T}(\btau)}{\dv}_{\partial T} 
   &&(\btau\in\Hdiv{\cT},\ \dv\in H^1(\cT)).
\end{align}
They give rise to the trace spaces
\[
   \HOtr{\cS} := \traceO{}(H^1_0(\Omega)),\quad
   \Hdivtr{\cS} := \tracediv{}(\Hdiv{\Omega}),
\]
furnished with the following norms,
\begin{align*}
   &\|\tv\|_{1/2,\cS} := \inf\{\|v\|_1;\; v\in H^1(\Omega),\; \traceO{}(v)=\tv\},\\
   &\|\tv\|_{(\div,\cT)'}
   :=
   \sup_{0\not=\dbtau\in\Hdiv{\cT}}
   \frac{\dual{\tv}{\dbtau}_\cS}{\|\dbtau\|_{\div,\cT}}
   \qquad (\tv\in\HOtr{\cS})
\end{align*}
and
\begin{align*}
   &\|\ttau\|_{-1/2,\cS}
   := \inf\{\|\btau\|_{\div};\; \btau\in\Hdiv{\Omega},\; \tracediv{}(\btau)=\ttau\},\\
   &\|\ttau\|_{(1,\cT)'}
   :=
   \sup_{0\not=\dv\in H^1(\cT)}
   \frac{\dual{\ttau}{\dv}_\cS}{\|\dv\|_{1,\cT}}
   \qquad (\ttau\in\Hdivtr{\cS}).
\end{align*}
Here, $\dual{\tv}{\dbtau}_\cS$ and $\dual{\ttau}{\dv}_\cS$ are the dualities
\begin{align} \label{dualg}
   \dual{\tv}{\dbtau}_\cS := \dual{\traceO{}(v)}{\dbtau}_\cS
   \quad\text{for any}\ v\in H^1_0(\Omega)\ \text{such that}\ \traceO{}(v)=\tv,
\end{align}
cf.~\eqref{dual_grad}, \eqref{dual_gradT}, and
\begin{align} \label{duald}
   \dual{\ttau}{\dv}_\cS := \dual{\tracediv{}(\btau)}{\dv}_\cS
   \quad\text{for any}\ \btau\in \Hdiv{\Omega}\ \text{such that}\ \tracediv{}(\btau)=\ttau,
\end{align}
cf.~\eqref{dual_div}, \eqref{dual_divT}.

Now, an ultraweak formulation of problem \eqref{Poisson}, with independent unknowns $\bsigma=\grad u$,
$\tu=\traceO{}(u)$, $\tsigma=\tracediv{}(\bsigma)$, is as follows. \emph{Find
\(\ds
   \bu:=(u,\bsigma,\tu,\tsigma)\in U:= L_2(\Omega)\times \bL_2(\Omega)\times \HOtr{\cS}\times \Hdivtr{\cS}
\)
such that}
\begin{align} \label{VF_P}
   b(\bu,\bv) :=
     \vdual{u}{\div\btau+\gamma v}_\cT + \vdual{\bsigma}{\btau+\grad v}_\cT
     - \dual{\tu}{\btau}_\cS - \dual{\tsigma}{v}_\cS = L(\bv) := \vdual{f}{v}
\end{align}
\emph{for any}
\(\ds
     \bv=(v,\btau)\in V:= H^1(\cT)\times \Hdiv{\cT},
\)
in operator form:
\[
   \bu=(u,\bsigma,\tu,\tsigma)\in U:\quad B\bu=L\quad\text{in}\ V'.
\]
Here, $\bL_2(\Omega):=(L_2(\Omega))^n$ and the dualities $\dual{\tu}{\btau}_\cS$, $\dual{\tsigma}{v}_\cS$
are defined as in \eqref{dualg}, \eqref{duald}, respectively.

For our discrete scheme we continue to define the trial-to-test operator $\Theta:\;U\to V$ by
\[
   \ip{\Theta\bu}{v}_V = b(\bu,v)\quad\forall v\in V
\]
where
\[
   \ip{(v,\btau)}{(\dv,\dbtau)}_V :=
   \vdual{v}{\dv} + \vdual{\grad v}{\grad\dv}_\cT + \vdual{\btau}{\dbtau} + \vdual{\div\btau}{\div\dbtau}_\cT
\]
is the inner product in $V$.
The spaces $U$ and $V$ are furnished with the following norms,
\begin{align*}
   \|\bu\|_U &:= \bigl(\|u\|^2 + \|\bsigma\|^2 + \|\tu\|_{1/2,\cS}^2 + \|\tsigma\|_{-1/2,\cS}^2\bigr)^{1/2}
   && (\bu=(u,\bsigma,\tu,\tsigma)\in U),\\
   \|\bv\|_V &:= \bigl(\|v\|_{1,\cT}^2 + \|\btau\|_{\mathrm{div},\cT}^2\bigr)^{1/2}
   && (\bv=(v,\btau)\in V).
\end{align*}
The standard DPG method is as follows. For a given finite-dimensional subspace $U_h\subset U$,
\begin{align} \label{DPG_P}
   \bu_h\in U_h:\quad b(\bu_h,\bv) = L(\bv)\quad\forall \bv\in \Theta(U_h).
\end{align}

\begin{theorem} \label{thm_P}
Let $f\in L_2(\Omega)$ and $\gamma\ge 0$ be given.
Formulation \eqref{VF_P} is well posed and equivalent to problem \eqref{Poisson}. Specifically,
if $u\in H^1_0(\Omega)$ solves \eqref{Poisson} then
$\bu=(u,\bsigma,\traceO{}(u),\tracediv{}(\bsigma))\in U$ solves \eqref{VF_P}
with $\bsigma:=\grad u$ on $\Omega$. Correspondingly, the solution
$\bu=(u,\bsigma,\tu,\tsigma)\in U$ of \eqref{VF_P} satisfies
$\bsigma=\grad u$, $\tu=\traceO{}(u)$, $\tsigma=\tracediv{}(\bsigma)$, and $u$ solves \eqref{Poisson}.

Furthermore, there exists a unique solution $\bu_h$ to \eqref{DPG_P}. It satisfies
\[
   \|\bu-\bu_h\|_U \le C\,\mathrm{inf}\,\{\|\bu-\bw\|_U;\; \bw\in U_h\}
\]
where $\bu\in U$ is the solution to \eqref{VF_P} and $C>0$ is a constant that is
independent of the mesh $\cT$ and the approximation space $U_h$.
\end{theorem}

\begin{proof}
For the case $\gamma=0$ this result is well known, see \cite{DemkowiczG_11_ADM}, and
an extension to $\gamma>0$ is straightforward.
\end{proof}

\begin{remark} \label{rem_P}
Main ingredients to prove Theorem~\ref{thm_P} are the stability of the adjoint problem
\begin{align} \label{adj_P}
   (v,\btau)\in H^1_0(\Omega)\times\Hdiv{\Omega}:\quad
   \div\btau+\gamma v=g_1\in L_2(\Omega),\quad
   \btau+\grad v=\bg_2\in\bL_2(\Omega),
\end{align}
the equivalence of the trace norms $\|\cdot\|_{1/2,\cS}\simeq\|\cdot\|_{(\div,\cT)'}$
and $\|\cdot\|_{-1/2,\cS}\simeq\|\cdot\|_{(1,\cT)'}$, and
the fact that the DPG method with exactly optimal test functions is a minimum
residual method with residual $B(\bu-\bu_h)$ measured in $V'$, cf.~{\rm\cite{DemkowiczG_11_CDP}}.
Carstensen \emph{et al.} {\rm\cite{CarstensenDG_16_BSF}} show that the equivalences of trace norms
hold with constants one,
\begin{align*}
   \|\tv\|_{1/2,\cS}=\|\tv\|_{(\div,\cT)'}\quad\forall\tv\in \HOtr{\cS},\quad
   \|\ttau\|_{-1/2,\cS}=\|\ttau\|_{(1,\cT)'}\quad\forall\ttau\in \Hdivtr{\cS}.
\end{align*}
Reducing \eqref{adj_P} to
\begin{equation} \label{adj_P_v}
   v\in H^1_0(\Omega):\quad -\Delta v+\gamma v=g_1-\div\bg_2\in \bigl(H^1_0(\Omega)\bigr)'
\end{equation}
it is clear that $\gamma>0$ leads to the robust control
$\gamma \|v\|^2+\|\grad v\|^2\le \gamma^{-1}\|g_1\|^2+\|\bg_2\|^2$ and, consequently, to the robust bound
\begin{align} \label{robust_P}
   \|u\|^2 + \|\bsigma\|^2 \le C \|B\bu\|_{V'}^2
   \quad\forall \bu=(u,\bsigma,\tu,\tsigma)\in U
\end{align}
with constant $C=C(\gamma)>0$ independent of $\Omega$.

On the other hand, in the case $\gamma=0$, a control of $\|v\|$ can only be obtained through
the Poincar\'e--Friedrichs inequality
\begin{align} \label{PP} 
   \|v\|\le \cPP\, \|\grad v\|\quad\forall v\in H^1_D(\Omega)
\end{align}
with constant $\cPP\lesssim\diam(\Omega)$ and $H^1_D(\Omega)=H^1_0(\Omega)$,
thus losing robustness of estimate \eqref{robust_P}
for large domains. This lack of robustness also affects the quasi-optimal error estimate.
Of course, the square of the best possible constant in \eqref{PP} is the inverse of the minimum
eigenvalue of the Laplacian. For different boundary conditions, only the (homogeneous)
essential ones enter the Poincar\'e--Friedrichs inequality, indicated by the space $H^1_D(\Omega)$ above.
In the numerical section we consider different boundary conditions though, for simplicity,
our theoretical presentation assumes homogeneous Dirichlet conditions on the whole of the boundary.
\end{remark}

\subsection{Domain-robust DPG method} \label{sec_Pd}

We use the same spaces as defined before, but consider scaled norms in $V$ and, consequently in $U$.
Specifically, denoting $d:=\cPP$ with $\cPP$ from \eqref{PP},
\begin{align*}
   &\|v\|_{1,d,\omega}^2 := d^{-2}\|v\|_\omega^2 + \|\grad v\|_\omega^2,\quad
   \|\btau\|_{\div,d,\omega}^2 := \|\btau\|_\omega^2 + d^2\|\div\btau\|_\omega^2
\end{align*}
($v\in H^1(\omega),\ \btau\in \Hdiv{\omega},\ \omega\subset\Omega$, and dropping the index $\omega$
when $\omega=\Omega$),
\begin{align*}
   &\|v\|_{1,d,\cT}^2 := \sum_{T\in\cT} \|v\|_{1,d,T}^2 \quad (v\in H^1(\cT)),\quad
   \|\btau\|_{\div,d,\cT}^2 := \sum_{T\in\cT} \|\btau\|_{\div,d,T}^2\quad (\btau\in\Hdiv{\cT}),
   \\
   &\|\tv\|_{1/2,d,\cS} := \inf\{\|v\|_{1,d};\; v\in H^1(\Omega),\; \traceO{}(v)=\tv\},\\
   &\|\tv\|_{(\div,d,\cT)'}
   :=
   \sup_{0\not=\dbtau\in\Hdiv{\cT}}
   \frac{\dual{\tv}{\dbtau}_\cS}{\|\dbtau\|_{\div,d,\cT}}
   \qquad (\tv\in\HOtr{\cS}),
   \\
   &\|\ttau\|_{-1/2,d,\cS}
   := \inf\{\|\btau\|_{\div,d};\; \btau\in\Hdiv{\Omega},\; \tracediv{}(\btau)=\ttau\},\\
   &\|\ttau\|_{(1,d,\cT)'}
   :=
   \sup_{0\not=\dv\in H^1(\cT)}
   \frac{\dual{\ttau}{\dv}_\cS}{\|\dv\|_{1,d,\cT}}
   \qquad (\ttau\in\Hdivtr{\cS}).
\end{align*}
\begin{prop} \label{prop_norms_Pd}
\begin{align*}
   \|\tv\|_{1/2,d,\cS}=\|\tv\|_{(\div,d,\cT)'}\quad\forall\tv\in \HOtr{\cS},\quad
   \|\ttau\|_{-1/2,d,\cS}=\|\ttau\|_{(1,d,\cT)'}\quad\forall\ttau\in \Hdivtr{\cS}.
\end{align*}
\end{prop}

\begin{proof}
By definition of the scaled norms it is clear that
\[
   \dual{\traceO{}(v)}{\dbtau}_\cS \le \|v\|_{1,d}\|\dbtau\|_{\div,d,\cT}
   \quad\forall v\in H^1(\Omega),\ \dbtau\in \Hdiv{\cT}
\]
and
\[
   \dual{\tracediv{}(\btau)}{\dv}_\cS \le \|\btau\|_{\div,d}\|\dv\|_{1,d,\cT}
   \quad\forall \btau\in \Hdiv{\Omega},\ \dv\in H^1(\cT),
\]
so that $\|\tv\|_{(\div,d,\cT)'}\le \|\tv\|_{1/2,d,\cS}$ and
$\|\ttau\|_{(1,d,\cT)'}\le \|\ttau\|_{-1/2,d,\cS}$
for any $\tv\in\HOtr{\cS}$ and $\ttau\in\Hdivtr{\cS}$.

To show the other direction we use the setting from \cite{FuehrerHN_19_UFK} based on inner products,
now with scalings. Of course, this setting (without scaling) is equivalent to the one
from \cite{CarstensenDG_16_BSF} in operator form.

Let $\tv\in\HOtr{\cS}$ be given. We define
$\btau\in\Hdiv{\cT}$ as the solution to
\begin{equation} \label{tv_tau}
   \vdual{\btau}{\dbtau} + d^2 \vdual{\div\btau}{\div\dbtau}_\cT = d^2 \dual{\tv}{\dbtau}_\cS
   \quad\forall \dbtau\in \Hdiv{\cT},
\end{equation}
and then $v\in H^1(\cT)$ as the solution to
\begin{equation} \label{tv_v}
   d^{-2}\vdual{v}{\dv} + \vdual{\grad v}{\grad\dv}_\cT = d^{-2} \dual{\traceO{}(\dv)}{\btau}_\cS
   \quad\forall \dv\in H^1(\cT).
\end{equation}
Noting that $\btau=d^2\grad_h\div_h\btau$ it follows that $v=\div_h\btau$
(by checking that $v:=\div_h\btau$ solves \eqref{tv_v}) and, also using \eqref{tv_tau},
\begin{align*}
   \dual{\traceO{}(v)}{\dbtau}_\cS
   &=
   \vdual{v}{\div\dbtau}_\cT + \vdual{\grad v}{\dbtau}_\cT
   =
   \vdual{\div\btau}{\div\dbtau}_\cT + d^{-2} \vdual{\btau}{\dbtau}_\cT
   \\
   &=
   \dual{\tv}{\dbtau}_\cS\quad\forall \dbtau\in \Hdiv{\cT}.
\end{align*}
Therefore, $\traceO{}(v)=\tv$ so that $v\in H^1_0(\Omega)$.
Then, selecting $\dbtau:=\btau$ in \eqref{tv_tau} and $\dv:=v$ in \eqref{tv_v}, we obtain
\[
   \|\btau\|_{\div,d,\cT}^2 = d^2 \dual{\tv}{\btau}_\cS = d^4 \|v\|_{1,d}^2.
\]
Noting that $\|v\|_{1,d}=\|\tv\|_{1/2,d,\cS}$ since $d^{-2} v-\div_h\grad_h v=0$, we conclude that
\[
   \|\tv\|_{1/2,d,\cS} = \frac {\dual{\tv}{\btau}_\cS}{\|\btau\|_{\div,d,\cT}}
   \le \|\tv\|_{(\div,d,\cT)'},
\]
which proves the first norm identity.

Now, let $\ttau\in \Hdivtr{\cS}$ be given. We have already seen that
$\|\ttau\|_{(1,d,\cT)'}\le \|\ttau\|_{-1/2,d,\cS}$. To show the other inequality we proceed
as before. We define $v\in H^1(\cT)$ as the solution to
\begin{equation} \label{ttau_v}
   d^{-2}\vdual{v}{\dv} + \vdual{\grad v}{\grad\dv}_\cT = \dual{\ttau}{\dv}_\cS
   \quad\forall \dv\in H^1(\cT),
\end{equation}
and $\btau\in\Hdiv{\cT}$ as the solution to
\begin{equation} \label{ttau_tau}
   \vdual{\btau}{\dbtau} + d^2 \vdual{\div\btau}{\div\dbtau}_\cT = \dual{\tracediv{}(\dbtau)}{v}_\cS
   \quad\forall \dbtau\in \Hdiv{\cT}.
\end{equation}
It follows that $\btau=\grad_h v$ and 
\begin{align*}
   \dual{\tracediv{}(\btau)}{\dv}_\cS
   &=
   \vdual{\btau}{\grad\dv}_\cT + \vdual{\div\btau}{\dv}_\cT
   =
   \vdual{\grad v}{\grad\dv}_\cT + d^{-2} \vdual{v}{\dv}_\cT
   \\
   &=
   \dual{\ttau}{\dv}_\cS\quad\forall \dv\in H^1(\cT).
\end{align*}
We conclude that $\tracediv{}(\btau)=\ttau$ so that $\btau\in\Hdiv{\Omega}$ and
$\|\ttau\|_{-1/2,d,\cS}=\|\btau\|_{\div,d}$. Therefore, setting
$\dv:=v$ in \eqref{ttau_v} and $\dbtau:=\btau$ in \eqref{ttau_tau}, we obtain
\[
   \|\ttau\|_{-1/2,d,\cS} = \frac {\dual{\ttau}{v}_\cS}{\|v\|_{1,d}} \le \|\ttau\|_{(1,d,\cT)'}.
\]
This finishes the proof.
\end{proof}

Collecting scalings, the spaces $U$ and $V$ are furnished with the norms
\begin{align*}
   \|\bu\|_{U,d}
   &:= \bigl(d^{-2}\|u\|^2 + \|\bsigma\|^2 + \|\tu\|_{1/2,d,\cS}^2 + \|\tsigma\|_{-1/2,d,\cS}^2\bigr)^{1/2}
   && (\bu=(u,\bsigma,\tu,\tsigma)\in U),\\
   \|\bv\|_{V,d} &:= \bigl(\|v\|_{1,d,\cT}^2 + \|\btau\|_{\mathrm{div},d,\cT}^2\bigr)^{1/2}
   && (\bv=(v,\btau)\in V).
\end{align*}
The domain-robust DPG scheme is only subtly different from \eqref{DPG_P}:
For a given finite-dimensional subspace $U_h\subset U$,
\begin{align} \label{DPG_Pd}
   \bu_h\in U_h:\quad b(\bu_h,\bv) = L(\bv)\quad\forall \bv\in \Theta_d(U_h)
\end{align}
where the operator $\Theta_d:\;U\to V$ (``trial-to-test operator'') is defined as
\[
   \ip{\Theta_d\bu}{v}_{V,d} = b(\bu,v)\quad\forall v\in V
\]
with scaled inner product
\[
   \ip{(v,\btau)}{(\dv,\dbtau)}_{V,d} :=
   d^{-2} \vdual{v}{\dv} + \vdual{\grad v}{\grad\dv}_\cT
   + \vdual{\btau}{\dbtau} + d^2 \vdual{\div\btau}{\div\dbtau}_\cT.
\]

\begin{theorem} \label{thm_Pd}
Let $f\in L_2(\Omega)$ be given and $\gamma=0$.
Assuming that \eqref{PP} holds, and the scaling parameter is chosen as $d=\cPP$
with $\cPP$ as in \eqref{PP}, the solution $\bu_h$ to \eqref{DPG_Pd} satisfies
\[
   \|\bu-\bu_h\|_{U,d} \le 9\,\mathrm{inf}\,\{\|\bu-\bw\|_{U,d};\; \bw\in U_h\}
\]
where $\bu\in U$ is the solution to \eqref{VF_P}.
\end{theorem}

\begin{proof}
Our proof uses the abstract framework from \cite{CarstensenDG_16_BSF}, with key ingredients
as discussed in Remark~\ref{rem_P}.

We already have proved the trace norm identities in Proposition~\ref{prop_norms_Pd}.
We need to check the stability of the adjoint problem \eqref{adj_P} with respect to the
scaled norm $\|\cdot\|_{V,d}$. Indeed, with $\gamma=0$, \eqref{adj_P_v} and \eqref{PP} imply that
\[
   \|\grad v\|^2 \le \bigl(d^2\|g_1\|^2 + \|\bg_2\|^2\bigr)^{1/2} \|v\|_{1,d}
   \ \le \sqrt{2} \bigl(d^2\|g_1\|^2 + \|\bg_2\|^2\bigr)^{1/2} \|\grad v\|
\]
so that
\begin{equation} \label{adj_Pd_v}
   d^{-2}\|v\|^2 \le \|\grad v\|^2 \le 2 \bigl(d^2\|g_1\|^2 + \|\bg_2\|^2\bigr).
\end{equation}
Returning to problem \eqref{adj_P}, we find that
\[
   d^{-2}\|v\|^2 + \|\grad v\|^2 + \|\btau\|^2 + d^2\|\div\btau\|^2
   = d^{-2}\|v\|^2 + \|\bg_2\|^2 + 2\vdual{g_1}{v} + d^2 \|g_1\|^2.
\]
Bounding $\vdual{g_1}{v}$ with Young's inequality and parameter $\delta=1/\sqrt{2}$, and
using \eqref{adj_Pd_v}, we continue to bound
\[
   \|(v,\btau)\|_{V,d}^2 \le (3+2\delta+\delta^{-1}) d^2 \|g_1\|^2 + (3+2\delta)\|\bg_2\|^2
                         \le (3+2\sqrt{2}) \bigl(d^2 \|g_1\|^2 + \|\bg_2\|^2\bigr).
\]
We conclude that
\begin{align} \label{bzero_infsup_P}
   \bigl(d^{-2}\|u\|^2 + \|\bsigma\|^2\bigr)^{1/2}
   &=
   \sup_{0\not=g_1\in L_2(\Omega)\atop 0\not=\bg_2\in\bL_2(\Omega)}
   \frac {\vdual{u}{g_1} + \vdual{\bsigma}{\bg_2}}
         {\bigl(d^2\|g_1\|^2 + \|\bg_2\|^2\bigr)^{1/2}}
   \nonumber\\
   &\le
   (3+2\sqrt{2})^{1/2} \sup_{0\not=v\in H^1_0(\Omega)\atop 0\not=\btau\in\Hdiv{\Omega}}
   \frac {b_0((u,\bsigma);(v,\btau))}
         {\|(v,\btau)\|_{V,d}}
\end{align}
where $b_0((u,\bsigma);(v,\btau)):=\vdual{u}{\div\btau} + \vdual{\bsigma}{\btau+\grad v}$.
This proves that Assumption~3.1 in \cite{CarstensenDG_16_BSF} is satisfied
with constant $c_0=(3+2\sqrt{2})^{-1/2}$. The relation
\[
   (v,\btau)\in H^1_0(\Omega)\times \Hdiv{\Omega}
   \quad\Leftrightarrow\quad
   \dual{\tv}{\btau}_\cS + \dual{\ttau}{v}_\cS = 0
   \quad\forall (\tv,\ttau)\in \HOtr{\cS} \times \Hdivtr{\cS}
\]
for any $(v,\btau)\in V$ (defining the space $Y_0$ there) is true
by \cite[Theorem~2.3]{CarstensenDG_16_BSF}.
Furthermore, Proposition~\ref{prop_norms_Pd} means that
\cite[Assumption~3.2]{CarstensenDG_16_BSF} holds with constant $\hat c=1$.
Noting that
\[
    b_0((u,\bsigma);(v,\btau))
    \le \sqrt{3/2} \|\bu\|_{U,d} \|\bv\|_{V,d}
    \quad\forall (u,\bsigma)\in L_2(\Omega)\times\bL_2(\Omega),\;
    (v,\btau)\in H^1_0(\Omega)\times\Hdiv{\Omega},
\]
that is, $\|b_0\|=\sqrt{3/2}$ in the notation from \cite{CarstensenDG_16_BSF},
\cite[Theorem~3.3]{CarstensenDG_16_BSF} proves that
\[
  c_1 \|\bu\|_{U,d} \le \|B\bu\|_{(V,d)'} := \sup_{0\not=\bv\in V} \frac {b(\bu,\bv)}{\|\bv\|_{V,d}}
\]
with
\[
   c_1^{-2}= c_0^{-2} + \bigl(\|b_0\|/c_0 + 1\bigr)^2 
   < 22.
\]
Finally, bounding
\begin{equation} \label{b_bound_P}
    b(\bu,\bv) \le \sqrt{3} \|\bu\|_{U,d} \|\bv\|_{V,d}
    \quad\forall \bu\in U,\; \bv\in V
\end{equation}
so that $\|B\bu\|_{(V,d)'}\le \sqrt{3} \|\bu\|_{U,d}$ for any $\bu\in U$,
and recalling that the DPG scheme minimizes $\|B(\bu-\bu_h)\|_{(V,d)'}$, this
proves the stated error bound with constant bounded by $\sqrt{3\cdot 22}<9$.
\end{proof}

\begin{cor} \label{cor_Pd}
Let $f\in L_2(\Omega)$ be given and $\gamma=0$.
Assuming that \eqref{PP} holds, and the scaling parameter is chosen as $d=\cPP$
with $\cPP$ as in \eqref{PP}, the solution $\bu_h=(u_h,\bsigma_h,\tu_h,\tsigma_h)$
to \eqref{DPG_Pd} satisfies
\[
   \bigl(d^{-2}\|u-u_h\|^2 + \|\bsigma-\bsigma_h\|^2\bigr)^{1/2}
   \le 3\sqrt{2}\, \mathrm{inf}\,\{\|\bu-\bw\|_{U,d};\; \bw\in U_h\}
\]
where $\bu=(u,\bsigma,\tu,\tsigma)\in U$ is the solution to \eqref{VF_P}.
\end{cor}

\begin{proof}
The statement follows by combining the minimum residual property of the scheme
with upper bound \eqref{b_bound_P} and lower bound \eqref{bzero_infsup_P}, the latter meaning that
\begin{align*}
   (3+2\sqrt{2})^{-1/2}
   \bigl(d^{-2}\|u-u_h\|^2 + \|\bsigma-\bsigma_h\|^2\bigr)^{1/2}
   &\le
   \sup_{0\not=v\in H^1_0(\Omega)\atop 0\not=\btau\in\times\Hdiv{\Omega}}
   \frac {b_0((u-u_h,\bsigma-\bsigma_h);(v,\btau))}{\|(v,\btau)\|_{V,d}}
   \\
   &\le
   \|B(\bu-\bu_h)\|_{(V,d)'}.
\end{align*}
Together we obtain the claimed error estimate with constant
bounded by $\sqrt{3}\,\sqrt{3+2\sqrt{2}}<3\sqrt{2}$.
\end{proof}

\subsection{Fully discrete scheme} \label{sec_Ph}

In practice, the trial-to-test operator $\Theta_d$ has to be approximated. The corresponding
fully discrete DPG scheme is also called ``practical'' DPG method.
The usual strategy consists in replacing the test space $V$ by a finite-dimensional space $V_h$. 
Discrete stability is ensured by the existence of a Fortin operator $\Pi_F\colon V\to V_h$ with 
\begin{align*}
   \|\Pi_F\|_{\mathcal{L}(V,V)} =: C_F <\infty \quad\text{and}\quad b(\bu_h,\bv-\Pi_F\bv) = 0
   \quad\forall \bu_h\in U_h,\,\bv\in V. 
\end{align*}
We refer to~\cite{GopalakrishnanQ_14_APD} for further details.

Here we show that the Fortin operators for second-order problems
constructed in~\cite{GopalakrishnanQ_14_APD} are equally suited for our domain-robust DPG scheme.
They guarantee discrete stability and quasi-best approximability of the fully discrete DPG method.

The analysis requires to be specific. We consider simplicial meshes of shape-regular elements, i.e., 
\begin{align} \label{shape}
  \sup_{T\in\cT} \frac{\diam(T)^n}{|T|} < \infty,
\end{align}
and use the discrete spaces
\begin{align*}
  \cP^p(T) &:= \{v:\;T\to\R;\; v \text{ is a polynomial of degree } p\}\quad (T\in\cT), \\
  \cP^p(\cT) &:= \{v\in L_2(\Omega);\;v|_T \in \cP^p(T) \,\forall T\in\cT\}.
\end{align*}
For an element $T\in\cT$ let $\EE_T$ denote its boundary elements (edges for $n=2$, faces for $n=3$)
and set
\begin{align*}
  \cP^p(E) &:= \{v:\;E\to\R;\; v \text{ is polynomial of degree } p\}\quad (E\in\EE_T,\ T\in\cT), \\
  \cP^p(\EE_T) &:= \{v\in L_2(\partial T);\; v|_E \in \cP^p(E) \,\forall E\in\EE_T\}\quad (T\in\cT).
\end{align*}
For an element $T\in\cT$, $\bn_T$ denotes the unit exterior normal vector along $\partial T$.

Now, for the Poisson problem, we consider the spaces
\begin{align*}
  U_h &:= \cP^p(\cT) \times \cP^p(\cT)^n \times
          \traceO{}(U_h^\mathrm{grad}(\cT)) \times \tracediv{}(U_h^\mathrm{div}(\cT)), \\
  V_h &:= \cP^{p+\Delta p_1}(\cT) \times \cP^{p+\Delta p_2}(\cT) 
\end{align*}
where $\Delta p_1 = n$, $\Delta p_2 = 2$ and 
\begin{align*}
  U_h^\mathrm{grad}(\cT)
  &:= \{u\in H_0^1(\Omega);\; u|_{\partial T} \in \cP^{p+1}(\EE_T) \, \forall T\in\cT\}, \\
  U_h^\mathrm{div}(\cT)
  &:= \{\bsigma\in \Hdiv{\Omega};\;
        \bsigma\cdot\bn_T|_{\partial T} \in \cP^p(\EE_T) \, \forall T\in  \cT\}.
\end{align*}
The fully discrete variant of our domain-robust DPG scheme then is
\begin{align} \label{DPG_Ph}
   \bu_h\in U_h:\quad b(\bu_h,\bv) = L(\bv)\quad\forall \bv\in \Theta_{d,h}(U_h)
\end{align}
with trial-to-test operator $\Theta_{d,h}:\;U\to V_h$ defined by
\[
   \ip{\Theta_{d,h}\bu}{v}_{V,d} = b(\bu,v)\quad\forall v\in V_h.
\]

\begin{theorem}\label{thm_Pdh}
Suppose that $\max_{T\in\cT}h_T\lesssim d:=d(\Omega)$. Theorem~\ref{thm_Pd} holds true
for scheme \eqref{DPG_Ph}. In particular, it exists a Fortin operator
$\Pi_F:\;(V,\,\|\cdot\|_{V,d})\to (V_h,\,\|\cdot\|_{V,d})$ that is bounded uniformly in $h$ and $d$
with norm $C_F\in\R$, and the following estimate holds,
\begin{align*}
   \|\bu-\bu_h\|_{U,d} \le 9\,C_F\,\mathrm{inf}\,\{\|\bu-\bw\|_{U,d};\; \bw\in U_h\}.
\end{align*}
\end{theorem}

\begin{proof}
We use the Fortin operator constructed in \cite{GopalakrishnanQ_14_APD}, and only have to
check that estimates hold uniformly in $d$ when using norm $\|\cdot\|_{V,d}$ instead of the
standard norm $\|\cdot\|_V$.

We denote $\Pi_F\bv:=(\Pi_{p+n}^\mathrm{grad}v,\Pi_{p+2}^\mathrm{div}\btau)$ for $\bv=(v,\btau)$
where $\Pi_{p+n}^\mathrm{grad}$ and $\Pi_{p+2}^\mathrm{div}$ are the operators defined
in~\cite[Section~3.4]{GopalakrishnanQ_14_APD}. 
As shown in \cite[Lemmas~3.2,~3.3]{GopalakrishnanQ_14_APD}, $\Pi_F\colon V\to V_h$ satisfies
\begin{align*}
   b(\bu_h,\bv-\Pi_F\bv) = 0 \quad\forall \bu_h\in U_h,\,\bv\in V.
\end{align*}
To prove that $\Pi_F$ is bounded uniformly in $h$ and $d(\Omega)$, let $T\in\cT$ be given.
From the proof of~\cite[Lemma~3.2]{GopalakrishnanQ_14_APD} we conclude with
$h_T\lesssim d$ that
\begin{align*}
    d^{-1}\|\Pi_{p+n}^\mathrm{grad}v\|_T &\lesssim d^{-1}\|v\|_T + h_T d^{-1}\|\nabla v\|_T
    \lesssim d^{-1}\|v\|_T + \|\nabla v\|_T
\end{align*}
and
\begin{align*}
    \|\nabla\Pi_{p+n}^\mathrm{grad}v\|_T \lesssim \|\nabla v\|_T.
\end{align*}
Similarly, we conclude from the proof of~\cite[Lemma~3.3]{GopalakrishnanQ_14_APD} that
\begin{align*}
    \|\Pi_{p+2}^\mathrm{\div}\btau\|_T
    \lesssim \|\btau\|_T + h_T\|\div\btau\|_T \lesssim \|\btau\|_T + d\|\div\btau\|_T
\end{align*}
and 
\begin{align*}
    d \|\div\Pi_{p+2}^\mathrm{\div}\btau\|_T \le d\|\div\btau\|_T.
\end{align*}
Combining the last four estimates proves that $\Pi_F:\;V\to V_h$ is bounded independently of $h$ and $d$
when using the scaled norm $\|\cdot\|_{V,d}$.
Finally, the quasi-best approximation property follows from the existence of a Fortin operator,
see~\cite[Theorem~2.1]{GopalakrishnanQ_14_APD}.
\end{proof}

\begin{remark}
The same argumentation as in the last result shows that Corollary~\ref{cor_Pd}
holds true for \eqref{DPG_Ph} when replacing constant $3\sqrt{2}$ with $3\sqrt{2}\,C_F$.
\end{remark}

\section{Kirchhoff--Love plate bending model} \label{sec_KL}

As in the previous section, we continue to consider a bounded, polygonal Lipschitz domain
$\Omega\subset\R^2$ (now only in two dimensions) with boundary $\partial\Omega$.

Our simplified (scaled) version of the Kirchhoff--Love plate bending model is
\begin{subequations} \label{KL}
\begin{alignat}{2}
     -\div\Div\MM           &= f  && \quad\text{in} \quad \Omega\label{KL1},\\
    \MM + \cC\Ggrad u &= 0  && \quad\text{in} \quad \Omega\label{KL2}\\
    \text{with}\quad
    u = 0,\quad \grad u &= 0 &&\quad\text{on}\quad\partial\Omega\quad\text{(clamped)}\label{KLc}\\
    \text{or}\qquad\qquad\qquad
    u = 0,\quad \bn\cdot\MM\bn &= 0
    &&\quad\text{on}\quad\partial\Omega\quad\text{(simply supported)}.\label{KLs}
\end{alignat}
\end{subequations}
Here, $\Omega$ is the mid-surface of the plate, $f$ the (scaled) transversal
bending load, $u$ the transverse deflection, $\MM$ the (scaled) bending moment tensor,
$\bn$ the exterior unit normal vector along $\partial\Omega$,
and $\Grad$ the symmetric gradient,
$\Grad{\bpsi}:=\Grad(\bpsi):=\frac 12(\grad\bpsi+(\grad\bpsi)^\transp)$ for a vector field $\bpsi$.
In particular, $\Ggrad u$ is the Hessian of $u$.
The operator $\div$ is the standard divergence, and $\Div$ is the row-wise divergence when
writing second-order tensors as $2\times 2$ matrix functions.
Furthermore, $\cC$ is a symmetric, positive definite tensor of order four.

Critical for stability of the problem is, as before, a Poincar\'e inequality. We
assume that there is a constant $\cPP$ that only depends on $\Omega$
and the chosen boundary condition, such that
\begin{equation} \label{PKL}
   \|v\|\le \cPP^2\, \|\Ggrad v\|\quad\forall v\in H^2(\Omega)\ \text{+BC}.
\end{equation}
Here ``+BC'' means $v\in H^2_0(\Omega)$ or $v\in H^1_0(\Omega)$ when considering \eqref{KLc}
or \eqref{KLs}, respectively. ($H^2(\Omega)$ and $H^2_0(\Omega)$ are the usual Sobolev spaces.)
Of course, when considering a combination of essential and natural boundary conditions,
only the essential ones enter condition \eqref{PKL}. For simplicity we present our
analysis for the boundary conditions \eqref{KLc}, \eqref{KLs}. An extension to different
types is not difficult and will be considered in the numerical section.

\subsection{Traces with scaled norms} \label{sec_KL_traces}

Let us introduce some further spaces.
For $\omega\subset\Omega$, we additionally need the space $\LL_2^s(\omega)$ of
symmetric tensors of order two with $L_2(\Omega)$ components, and the standard spaces
$H^2(\omega)$, $H^2_0(\omega)$, and $H^2_s(\omega):=H^2(\omega)\cap H^1_0(\omega)$. 
Also, $\HdDiv{\omega}$ is the space of $\LL_2^s(\omega)$ tensors $\QQ$ with
$\dDiv\QQ\in L_2(\omega)$.
We use the norms
\begin{align*}
   \|v\|_{2,d,\omega}^2 &:= d^{-4} \|v\|_\omega^2 + \|\Ggrad v\|_\omega^2\quad (v\in H^2(\omega)),\\
   \|\QQ\|_{\dDiv,d,\omega}^2 &:= \|\QQ\|_\omega^2 + d^4 \|\dDiv\QQ\|_\omega^2\quad (\QQ\in\HdDiv{\omega})
\end{align*}
and drop the index $\omega$ when $\omega=\Omega$.
Here, $d:=\cPP$ is the parameter $\cPP$ from \eqref{PKL}.
For a mesh $\cT$ of elements $\{T\}$ as before we introduce the product spaces
\[
   H^2(\cT):=\Pi_{T\in\cT} H^2(T),\qquad \HdDiv{\cT}:=\Pi_{T\in\cT} \HdDiv{T}
\]
with norms $\|\cdot\|_{2,d,\cT}$ and $\|\cdot\|_{\dDiv,d,\cT}$, respectively.
There are two canonical (local) trace operators,
$\traceGG{T}:\;H^2(T)\to (\HdDiv{T})'$ and $\traceDD{T}:\;\HdDiv{T}\to (H^2(T))'$
with support on $\partial T$ ($T\in\cT$), defined by
\begin{align}
   \label{dual_GgradT}
   &\dual{\traceGG{T}(v)}{\dQQ}_{\partial T} :=
   \vdual{v}{\dDiv\dQQ}_T - \vdual{\Ggrad v}{\dQQ}_T
   &&(v\in H^2(T),\ \dQQ\in\HdDiv{T}),\\
   \label{dual_dDivT}
   &\dual{\traceDD{T}(\QQ)}{\dv}_{\partial T} :=
   \vdual{\dDiv\QQ}{\dv}_T - \vdual{\QQ}{\Ggrad \dv}_T
   &&(\QQ\in\HdDiv{T},\ \dv\in H^2(T)),
\end{align}
and the product variants $\traceGG{}:\;H^2(\cT)\to (\HdDiv{\cT})'$,
$\traceDD{}:\;\HdDiv{\cT}\to (H^2(\cT))'$ defined by
\begin{align}
   \label{dual_Ggrad}
   &\dual{\traceGG{}(v)}{\dQQ}_\cS := \sum_{T\in\cT} \dual{\traceGG{T}(v)}{\dQQ}_{\partial T}
   &&(v\in H^2(\cT),\ \dQQ\in\HdDiv{\cT}),\\
   \label{dual_dDiv}
   &\dual{\traceDD{}(\QQ)}{\dv}_\cS := \sum_{T\in\cT} \dual{\traceDD{T}(\QQ)}{\dv}_{\partial T} 
   &&(\QQ\in\HdDiv{\cT},\ \dv\in H^2(\cT)).
\end{align}
They give rise to the trace spaces
\[
   \bH^{3/2,1/2}(\cS) := \traceGG{}(H^2(\Omega)),\quad
   \bH^{-3/2,-1/2}(\cS) := \traceDD{}(\HdDiv{\Omega})
\]
and the respective subspaces
\[
   \bH^{3/2,1/2}_{0}(\cS) := \traceGG{}(H^2_s(\Omega)),\quad
   \bH^{3/2,1/2}_{00}(\cS) := \traceGG{}(H^2_0(\Omega))
\]
and
\[
   \bH^{-3/2,-1/2}_{0}(\cS) := \traceDD{}(\HdDivz{\Omega})
\]
where
\begin{equation} \label{HdDivz_def}
   \HdDivz{\Omega} := \{\QQ\in\HdDiv{\Omega};\; \dual{\traceDD{}(\QQ)}{\dv}_\cS=0
                                                    \ \forall \dv\in H^2_s(\Omega)\}.
\end{equation}
They are furnished with the following norms,
\begin{align*}
   &\|\btv\|_{3/2,1/2,d,\cS} := \inf\{\|v\|_{2,d};\; v\in H^2(\Omega),\; \traceGG{}(v)=\btv\},\\
   &\|\btv\|_{(\dDiv,d,\cT)'}
   :=
   \sup_{0\not=\dQQ\in\HdDiv{\cT}}
   \frac{\dual{\btv}{\dQQ}_\cS}{\|\dQQ\|_{\dDiv,d,\cT}}
   \qquad (\btv\in\bH^{3/2,1/2}(\cS))
\end{align*}
and
\begin{align*}
   &\|\tQ\|_{-3/2,-1/2,d,\cS}
   := \inf\{\|\QQ\|_{\dDiv,d};\; \QQ\in\HdDiv{\Omega},\; \traceDD{}(\QQ)=\tQ\},\\
   &\|\tQ\|_{(2,d,\cT)'}
   :=
   \sup_{0\not=\dv\in H^2(\cT)}
   \frac{\dual{\tQ}{\dv}_\cS}{\|\dv\|_{2,d,\cT}}
   \qquad (\tQ\in \bH^{-3/2,-1/2}(\cS)).
\end{align*}
Here, $\dual{\btv}{\dQQ}_\cS$ and $\dual{\tQ}{\dv}_\cS$ are the dualities, respectively,
\begin{align} \label{dualGG}
   \dual{\btv}{\dQQ}_\cS := \dual{\traceGG{}(v)}{\dQQ}_\cS
   \quad\text{for any}\ v\in H^2(\Omega)\ \text{such that}\ \traceGG{}(v)=\btv,
\end{align}
cf.~\eqref{dual_Ggrad}, \eqref{dual_GgradT}, and
\begin{align} \label{dualDD}
   \dual{\tQ}{\dv}_\cS := \dual{\traceDD{}(\QQ)}{\dv}_\cS
   \quad\text{for any}\ \QQ\in \HdDiv{\Omega}\ \text{such that}\ \traceDD{}(\QQ)=\tQ,
\end{align}
cf.~\eqref{dual_dDiv}, \eqref{dual_dDivT}.

The following are the scaled versions of \cite[Propositions~3.5,~3.9]{FuehrerHN_19_UFK},
though for the (more general) traces not involving boundary conditions.

\begin{prop} \label{prop_norms_KLd}
\begin{align*}
   \|\btv\|_{3/2,1/2,d,\cS}&=\|\btv\|_{(\dDiv,d,\cT)'}
   &&\hspace{-6em}\forall\btv\in \bH^{3/2,1/2}(\cS),\\
   \|\tQ\|_{-3/2,-1/2,d,\cS}&=\|\tQ\|_{(2,d,\cT)'}
   &&\hspace{-6em}\forall\tQ\in \bH^{-3/2,-1/2}(\cS).
\end{align*}
\end{prop}

\begin{proof}
The inequalities ``$\ge$'' are immediately clear by bounding the dualities
\[
   \dual{\traceGG{}(v)}{\dQQ}_\cS \le \|v\|_{2,d}\|\dQQ\|_{\dDiv,d,\cT}
   \quad\forall v\in H^2(\Omega),\ \dQQ\in \HdDiv{\cT}
\]
and
\[
   \dual{\traceDD{}(\QQ)}{\dv}_\cS \le \|\QQ\|_{\dDiv,d}\|\dv\|_{2,d,\cT}
   \quad\forall \QQ\in \HdDiv{\Omega},\ \dv\in H^2(\cT).
\]
To show the other direction we proceed as
in the proofs of \cite[Lemmas~3.2, 3.3]{FuehrerHN_19_UFK}, but considering the
scalings as before for the Poisson problem.

Let $\btv\in\bH^{3/2,1/2}(\cS)$ be given. We define
$\QQ\in\HdDiv{\cT}$ as the solution to
\begin{equation} \label{tv_Q}
   \vdual{\QQ}{\dQQ} + d^4 \vdual{\dDiv\QQ}{\dDiv\dQQ}_\cT = d^4 \dual{\btv}{\dQQ}_\cS
   \quad\forall \dQQ\in \HdDiv{\cT},
\end{equation}
and then $v\in H^2(\cT)$ as the solution to
\begin{equation} \label{tv_v2}
   d^{-4}\vdual{v}{\dv} + \vdual{\Ggrad v}{\Ggrad\dv}_\cT = d^{-4} \dual{\traceGG{}(\dv)}{\QQ}_\cS
   \quad\forall \dv\in H^2(\cT).
\end{equation}
Analogously as in the proof of Proposition~\ref{prop_norms_Pd} we conclude that
$v=\div_h\Div_h\QQ$ and $\traceGG{}(v)=\btv$ so that $v\in H^2(\Omega)$.
Furthermore, selecting $\dQQ:=\QQ$ in \eqref{tv_Q} and $\dv:=v$ in \eqref{tv_v2}, it follows that
\[
   \|\QQ\|_{\dDiv,d,\cT}^2 = d^4 \dual{\btv}{\QQ}_\cS = d^8 \|v\|_{2,d}^2
\]
and
\[
   \|\btv\|_{3/2,1/2,d,\cS} = \|v\|_{2,d} = \frac {\dual{\btv}{\QQ}_\cS}{\|\QQ\|_{\dDiv,d,\cT}}
   \le \|\btv\|_{(\dDiv,d,\cT)'}.
\]
The bound for $\tQ$ is also straightforward. Given $\tQ\in \bH^{-3/2,-1/2}(\cS)$ we define
$v\in H^2(\cT)$ as the solution to
\begin{equation} \label{tQ_v}
   d^{-4}\vdual{v}{\dv} + \vdual{\Ggrad v}{\Ggrad\dv}_\cT = \dual{\tQ}{\dv}_\cS
   \quad\forall \dv\in H^2(\cT),
\end{equation}
and $\QQ\in\HdDiv{\cT}$ as the solution to
\begin{equation} \label{tQ_tau}
   \vdual{\QQ}{\dQQ} + d^4 \vdual{\dDiv\QQ}{\dDiv\dQQ}_\cT = \dual{\traceDD{}(\dQQ)}{v}_\cS
   \quad\forall \dQQ\in \HdDiv{\cT}.
\end{equation}
Then $\QQ=-\Grad_h\grad_h v$, $\traceDD{}(\QQ)=\tQ$, $\QQ\in\HdDiv{\Omega}$ and, setting
$\dv:=v$ in \eqref{tQ_v} and $\dQQ:=\QQ$ in \eqref{tQ_tau}, we obtain
\[
   \|\tQ\|_{-3/2,-1/2,d,\cS} = \|\QQ\|_{\dDiv,d}
   = \frac {\dual{\tQ}{v}_\cS}{\|v\|_{2,d,\cT}} \le \|\tQ\|_{(2,d,\cT)'}.
\]
\end{proof}

\subsection{Domain-robust DPG method} \label{sec_KL_DPG}

For the clamped plate we consider the ultraweak formulation of \eqref{KL} studied in
\cite{FuehrerHN_19_UFK}. The simply supported case is taken from \cite{FuehrerHS_20_UFR}.
Depending on the boundary condition, we need different trace spaces. For the clamped plate
we define
\[
   \bH^{3/2,1/2}_{c}(\cS) := \bH^{3/2,1/2}_{00}(\cS),\quad
   \bH^{-3/2,-1/2}_{c}(\cS) := \bH^{-3/2,-1/2}(\cS),
\]
and for the simply supported case,
\[
   \bH^{3/2,1/2}_{s}(\cS) := \bH^{3/2,1/2}_{0}(\cS),\quad
   \bH^{-3/2,-1/2}_{s}(\cS) := \bH^{-3/2,-1/2}_{0}(\cS).
\]
Introducing the independent trace variables
$\tM:=\traceDD{}(\MM)$, $\btu:=\traceGG{}(u)$, and spaces
\begin{align*}
   &U_a := L_2(\Omega)\times\LL_2^s(\Omega)\times \bH^{3/2,1/2}_{a}(\cS) \times \bH^{-3/2,-1/2}_a(\cS)
   \quad (a\in\{c,s\}),\\
   &V := H^2(\cT)\times H(\div\Div\!,\cT)
\end{align*}
with respective scaled norms
\begin{align*}
   \|(u,\MM,\btu,\tM)\|_{U,d}^2
   &:=
   d^{-4}\|u\|^2 + \|\MM\|^2 + \|\btu\|_{3/2,1/2,d,\cS}^2 + \|\tM\|_{-3/2,-1/2,d,\cS}^2,
   \\
   \|(v,\QQ)\|_{V,d}^2
   &:=
   \|v\|_{2,d,\cT}^2 + \|\QQ\|_{\dDiv,d,\cT}^2,
\end{align*}
our variational formulation of \eqref{KL} is:
\emph{Find $(u,\MM,\btu,\tM)\in U$ such that}
\begin{align} \label{VF_KL}
   b(u,\MM,\btu,\tM;v,\QQ) = L(v,\QQ)
   \quad\forall (v,\QQ)\in V.
\end{align}
Here, $U=U_c$ when considering boundary condition \eqref{KLc} (clamped),
$U=U_s$ for the simply supported case \eqref{KLs},
\begin{align*}
   b(u,\MM,\btu,\tM;v,\QQ)
   :=
   &\vdual{\MM}{\Ggrad v+\cCinv\QQ}_\cT
   + \vdual{u}{\dDiv\QQ}_\cT
   - \dual{\btu}{\QQ}_\cS + \dual{\tM}{v}_\cS,
\end{align*}
$L(v,\QQ) := -\vdual{f}{v}$, and the dualities $\dual{\btu}{\QQ}_\cS$, $\dual{\tM}{v}_\cS$
are as defined in \eqref{dualGG},~\eqref{dualDD}, respectively.

The DPG scheme is standard, except for the use of a scaled inner product in $V$.
For a given finite-dimensional subspace $U_h\subset U$,
\begin{align} \label{DPG_KLd}
   \bu_h\in U_h:\quad b(\bu_h,\bv) = L(\bv)\quad\forall \bv\in \Theta_d(U_h)
\end{align}
where the operator $\Theta_d:\;U\to V$ is defined as
\[
   \ip{\Theta_d\bu}{v}_{V,d} = b(\bu,v)\quad\forall v\in V
\]
with scaled inner product
\[
   \ip{(v,\QQ)}{(\dv,\dQQ)}_{V,d} :=
   d^{-4} \vdual{v}{\dv} + \vdual{\Ggrad v}{\Ggrad\dv}_\cT
   + \vdual{\QQ}{\dQQ} + d^4 \vdual{\dDiv\QQ}{\dDiv\dQQ}_\cT.
\]

\begin{theorem} \label{thm_KLd}
Let $f\in L_2(\Omega)$ be given. Formulation \eqref{VF_KL}
is well posed and equivalent to problem \eqref{KL}. Specifically,
if $(u,\MM)\in H^2(\Omega)\times \HdDiv{\Omega}$ solves \eqref{KL}
with one of the boundary conditions \eqref{KLc} or \eqref{KLs}, then
$\bu=(u,\MM,\traceGG{}(u),\traceDD{}(\MM))\in U$ solves \eqref{VF_P}.
Correspondingly, the solution $\bu=(u,\MM,\btu,\tM)\in U$ of \eqref{VF_KL} satisfies
$\btu=\traceGG{}(u)$, $\tM=\traceDD{}(\MM)$, and $(u,\MM)\in H^2(\Omega)\times\HdDiv{\Omega}$
solves \eqref{KL} with the corresponding boundary condition.

Furthermore, there exists a unique solution $\bu_h$ to \eqref{DPG_KLd}.
Assuming that \eqref{PKL} holds, and the scaling parameter is chosen as
$d=\cPP$ with $\cPP$ from \eqref{PKL}, $\bu_h$ satisfies
\[
   \|\bu-\bu_h\|_{U,d} \le C\,\mathrm{inf}\,\{\|\bu-\bw\|_{U,d};\; \bw\in U_h\}
\]
with a constant $C>0$ that is independent of $f$, $U_h$, $\cT$ and $\Omega$.
Here, $\bu\in U$ is the solution to \eqref{VF_KL}.
\end{theorem}

\begin{proof}
The well-posedness of \eqref{VF_KL} and its equivalence with \eqref{KL} has been proved in
\cite{FuehrerHN_19_UFK} (for the clamped case) and in \cite{FuehrerHS_20_UFR} (for
the simply supported case). It remains to prove the quasi-optimal error estimate.
Again, we use the abstract framework from \cite{CarstensenDG_16_BSF}
to show the uniform bound
\[
   \|\bu\|_{U,d} \lesssim \sup_{0\not=\bv\in V} \frac {b(\bu,\bv)}{\|\bv\|_{V,d}}
   \quad\forall \bu\in U.
\]
The immediate uniform boundedness
\[
    b(\bu,\bv) \lesssim \|\bu\|_{U,d} \|\bv\|_{V,d} \quad\forall \bu\in U,\; \bv\in V
\]
then gives the uniform equivalence of $\|\cdot\|_{U,d}$ and the residual in
$V'$ with norm $\|\cdot\|_{V,d}$ in $V$ and, thus, the uniform quasi-optimal convergence
of the DPG scheme, as stated.

Key ingredients for the framework from \cite{CarstensenDG_16_BSF} are
provided by Proposition~\ref{prop_norms_KLd} and reference \cite{FuehrerHS_20_UFR}
which is based on \cite{FuehrerHN_19_UFK}.
Remaining ingredient is the stability with respect to the scaled norm of the following adjoint problem,
\begin{align} \label{adj_KL}
   (v,\QQ)\in V(\Omega):\quad
   \dDiv\QQ=g_1\in L_2(\Omega),\quad
   \cCinv\QQ+\Ggrad v=\GG_2\in\LL_2^s(\Omega).
\end{align}
Here,
$V(\Omega):=H^2_0(\Omega)\times\HdDiv{\Omega}$ for the clamped case and
$V(\Omega):=H^2_s(\Omega)\times\HdDivz{\Omega}$ for the simply supported boundary condition.

Eliminating $\QQ$ we obtain the reduced problem
\[
   v\in H^2_m(\Omega):\quad
   \vdual{\cC\Ggrad v}{\Ggrad\dv} = \vdual{\cC\GG_2}{\Ggrad\dv}-\vdual{g_1}{\dv}
   \quad\forall\dv\in H^2_m(\Omega)
\]
with $H^2_m(\Omega):=H^2_0(\Omega)$ in the clamped case and
$H^2_m(\Omega):=H^2_s(\Omega)$ in the simply supported case, cf.~\cite[(47)]{FuehrerHS_20_UFR}.
Noting that $\cC$ is symmetric and positive definite we deduce the bound
\[
   \|\Ggrad v\|^2 \lesssim \bigl(d^4\|g_1\|^2 + \|\GG_2\|^2\bigr)^{1/2} \|v\|_{2,d}.
\]
Recall that the inherent constant only depends on $\cC$.
The Poincar\'e bounds \eqref{PKL} imply that
\begin{equation} \label{adj_KLd_v}
   \|v\|_{2,d}^2 \lesssim d^4\|g_1\|^2 + \|\GG_2\|^2.
\end{equation}
Returning to \eqref{adj_KL} we continue to bound
\[
   d^4\|\dDiv\QQ\|^2 = d^4 \|g_1\|^2,\quad
   \|\QQ\|^2 \le 2\|\cC \GG_2\|^2 + 2\|\cC\Ggrad v\|^2
              \lesssim d^4 \|g_1\|^2 + \|\GG_2\|^2,
\]
and combination with \eqref{adj_KLd_v} yields
\[
   \|(v,\btau)\|_{V,d}^2 = \|v\|_{2,d}^2 + \|\btau\|_{\dDiv,d}^2
   \lesssim
   d^4 \|\g_1\|^2 + \|\GG_2\|^2.
\]
We conclude that
\begin{align} \label{bzero_infsup_KL}
   \bigl(d^{-4}\|u\|^2 + \|\MM\|^2\bigr)^{1/2}
   &=
   \sup_{0\not=g_1\in L_2(\Omega)\atop 0\not=\GG_2\in\LL_2^s(\Omega)}
   \frac {\vdual{u}{g_1} + \vdual{\MM}{\GG_2}}
         {\bigl(d^4\|g_1\|^2 + \|\GG_2\|^2\bigr)^{1/2}}
   \lesssim
   \sup_{0\not=v\in H^2_0(\Omega)\atop 0\not=\QQ\in\HdDiv{\Omega}}
   \frac {b_0((u,\MM);(v,\QQ))}
         {\|(v,\QQ)\|_{V,d}}
\end{align}
where $b_0((u,\MM);(v,\QQ)):=\vdual{u}{\dDiv\QQ} + \vdual{\MM}{\cCinv\QQ+\Ggrad v}$.
Noting that, by \cite[Proposition~13]{FuehrerHS_20_UFR},
\[
   (v,\QQ)\in H^2_0(\Omega)\times \Hdiv{\Omega}
   \ \Leftrightarrow\
   \dual{\btv}{\QQ}_\cS + \dual{\tQ}{v}_\cS = 0
   \ \forall (\btv,\tQ)\in \bH^{3/2,1/2}_{00}(\cS) \times \bH^{-3/2,-1/2}(\cS)
\]
and
\[
   (v,\QQ)\in H^2_s(\Omega)\times \HdDivz{\Omega}
   \ \Leftrightarrow\
   \dual{\btv}{\QQ}_\cS + \dual{\tQ}{v}_\cS = 0
   \ \forall (\btv,\tQ)\in \bH^{3/2,1/2}_{0}(\cS) \times \bH^{-3/2,-1/2}_{0}(\cS)
\]
for any $(v,\btau)\in V$, \eqref{bzero_infsup_KL} shows that Assumption~3.1 in
\cite{CarstensenDG_16_BSF} is satisfied.
Also, Proposition~\ref{prop_norms_KLd} means that Assumption~3.2 from
\cite{CarstensenDG_16_BSF} holds with constant $\hat c=1$.
Bounding
\[
    b_0((u,\MM);(v,\QQ))
    \lesssim \|\bu\|_{U,d} \|\bv\|_{V,d}
    \quad\forall (u,\MM)\in L_2(\Omega)\times\LL_2^s(\Omega),\;
    (v,\QQ)\in H^2(\Omega)\times\Hdiv{\Omega}
\]
this means that $\|b_0\|\lesssim 1$ in the notation of \cite{CarstensenDG_16_BSF},
so that \cite[Theorem~3.3]{CarstensenDG_16_BSF} proves that
\[
  \|\bu\|_{U,d} \lesssim \sup_{0\not=\bv\in V} \frac {b(\bu,\bv)}{\|\bv\|_{V,d}}.
\]
This finishes the proof of the theorem.
\end{proof}

\begin{remark}
In the trivial case of tensor $\cC$ being the identity, identical techniques as used in
the proofs of Theorem~\ref{thm_Pd} and Corollary~\ref{cor_Pd} show that
Theorem~\ref{thm_KLd} holds with constant $C=9$, and that the following improved bound
holds true,
\[
   \bigl(d^{-4}\|u-u_h\|^2 + \|\MM-\MM_h\|^2\bigr)^{1/2}
   \le 3\sqrt{2}\, \mathrm{inf}\,\{\|\bu-\bw\|_{U,d};\; \bw\in U_h\}.
\]
\end{remark}

\subsection{Fully discrete scheme} \label{sec_KLh}

We use the notation for discrete spaces introduced in \S\ref{sec_Ph}, and use triangular
meshes $\cT$ consisting of shape-regular elements, cf.~\eqref{shape}. For the fully discrete DPG scheme
we employ lowest-order discrete spaces
\begin{align*}
  U_h &:= \cP^0(\cT)\times \PP^{0,s}(\cT) \times
          \traceGG{}(U_h^\mathrm{Ggrad}(\cT)) \times \traceDD{}(U_h^\mathrm{dDiv}(\cT)),\\
  V_h &:= \cP^3(\cT)\times \PP^{4,s}(\cT).
\end{align*}
Here, 
\begin{align*}
   U_h^\mathrm{Ggrad}(\cT)
   &:=
   \{u\in H_0^2(\Omega);\; u|_{\partial T} \in P^3(\EE_T),\,
                           \partial_{\bn_T}u|_{\partial T} \in P^1(\EE_T)\ \forall T\in\cT\}, \\
   U_h^\mathrm{dDiv}(\cT)
   &:=\\&\hspace{-3em}
   \{\MM\in \LL_2^s(\Omega);\;
     \bn_T\cdot\MM\bn_T|_{\partial T} \in \cP^0(\EE_T),\,
     \bigl(\bn_T\cdot\Div\MM + \partial_{\bt_T}(\bt\cdot\MM\bn_T)\bigr)|_{\partial T}\in \cP^0(\EE_T)\
     \forall T\in\cT\},
\end{align*}
$\PP^{p,s}(\cT)$ is the space of $\LL_2^s(\Omega)$-tensors whose components are
piecewise polynomials of degree $p$. Furthermore, $\partial_{\bn_T}$ and $\partial_{\bt_T}$
denote the exterior normal and (mathematically positive) tangential derivative operators,
respectively, on $\partial T$ ($T\in\cT$).

Our fully discrete DPG scheme for the Kirchhoff--Love model is
\begin{align} \label{DPG_KLdh}
   \bu_h\in U_h:\quad b(\bu_h,\bv) = L(\bv)\quad\forall \bv\in \Theta_{d,h}(U_h)
\end{align}
with trial-to-test operator $\Theta_{d,h}:\;U\to V_h$ defined as
\[
   \ip{\Theta_{d,h}\bu}{v}_{V,d} = b(\bu,v)\quad\forall v\in V_h.
\]
This scheme converges quasi-optimally and robustly with respect to $\Omega$.

\begin{theorem}
Assume that \eqref{PKL} holds with parameter $\cPP$, and that
$\max_{T\in\cT}h_T\lesssim d:=\cPP$. Then Theorem~\ref{thm_KLd} holds true for scheme
\eqref{DPG_KLdh}. In particular, it exists a Fortin operator $\Pi_F:\;V\to V_h$
that is uniformly bounded in $h$ and $d$ with respect to the scaled norm
$\|\cdot\|_{V,d}$. Scheme \eqref{DPG_KLdh} converges robustly
quasi-uniformly, that is,
  \begin{align*}
   \|\bu-\bu_h\|_{U,d} \le C\,C_F\,\mathrm{inf}\,\{\|\bu-\bw\|_{U,d};\; \bw\in U_h\}
  \end{align*}
with constant $C>0$ from Theorem~\ref{thm_KLd} and norm $C_F$ of the Fortin operator.
\end{theorem}

The proof is analogous to the proof of Theorem~\ref{thm_Pdh}, by using the Fortin operator
from~\cite{FuehrerH_19_FDD} and checking its uniform boundedness with respect to the scaled norm.
This is similar to the estimates for the Poisson problem, and is omitted.

\section{Numerical experiments} \label{sec_num}

In this section we present numerical examples that support our theoretical findings
for the two model problems of a second and fourth order.
Throughout we consider the computational domain
\begin{align*}
  \Omega = (0,R_1)\times (0,R_2)
\end{align*}
triangulated with shape-regular elements. We use uniform meshes and
lowest order approximation spaces, as defined in \S\ref{sec_Ph} with $p=0$ for the Poisson
problem, and as specified in \S\ref{sec_KLh} for the plate bending model.
In all the plots, (black) dotted lines indicate the lines with theoretical order of convergence
$\OO(h) = \OO( (\#\cT)^{-1/2})$. 

\subsection{Poisson problem}\label{sec:exp:secondorder}

First let us note that the Poincar\'e--Friedrichs inequality \eqref{PP} holds in $H^1_0(\Omega)$
with $\cPP=\min\{R_1,R_2\}$, cf.~\cite[II.1.5]{Braess_97_FET}, which is our choice of $d$.

We consider the manufactured solution 
\begin{align*}
  u(x,y) = \sin(\pi x/R_1)\sin(\pi y/R_2),
\end{align*}
and define the right-hand side $f$ via
\begin{align*}
  -\Delta u + \gamma\, u = f \qquad (\gamma\in\{0,1\}).
\end{align*}
In the displayed results we use the notation
\begin{align*}
  \|\bu-\bu_h\|_{E,d} &:= \sup_{0\neq \bv\in V_h} \frac{b(\bu-\bu_h,\bv)}{\|\bv\|_{V,d}},\quad
  \|\bu-\bu_h\|_{E} := \|\bu-\bu_h\|_{E,1}.
\end{align*}
For the first set of experiments we choose $\gamma=0$ and consider $R=R_1=R_2\in\{1,10,100\}$.
The results are collected in Figure~\ref{fig:diffusion} where the first column corresponds to
DPG approximations using $\|\cdot\|_{V}$ as test norm and the second one to approximations using
$\|\cdot\|_{V,d}$. The $j$-th row corresponds to the value $R=10^j$.
As $R$ grows we observe that the standard DPG method develops a pre-asymptotic range of reduced
convergence. For $R=100$, there is basically no error reduction up to
$100,000$ unknowns, despite of a smooth solution. The standard DPG scheme exhibits locking.
Additionally, control of the error of the field variables by the residual is completely lost.
This has serious consequences on adaptivity steered by the residual (not studied here).
In contrast, our proposed DPG variant is robust with respect to $R$, all the error curves
are insensitive to changes of $R$.
(We also calculated for $R=10^3$ and $R=10^4$ with results not shown here.)

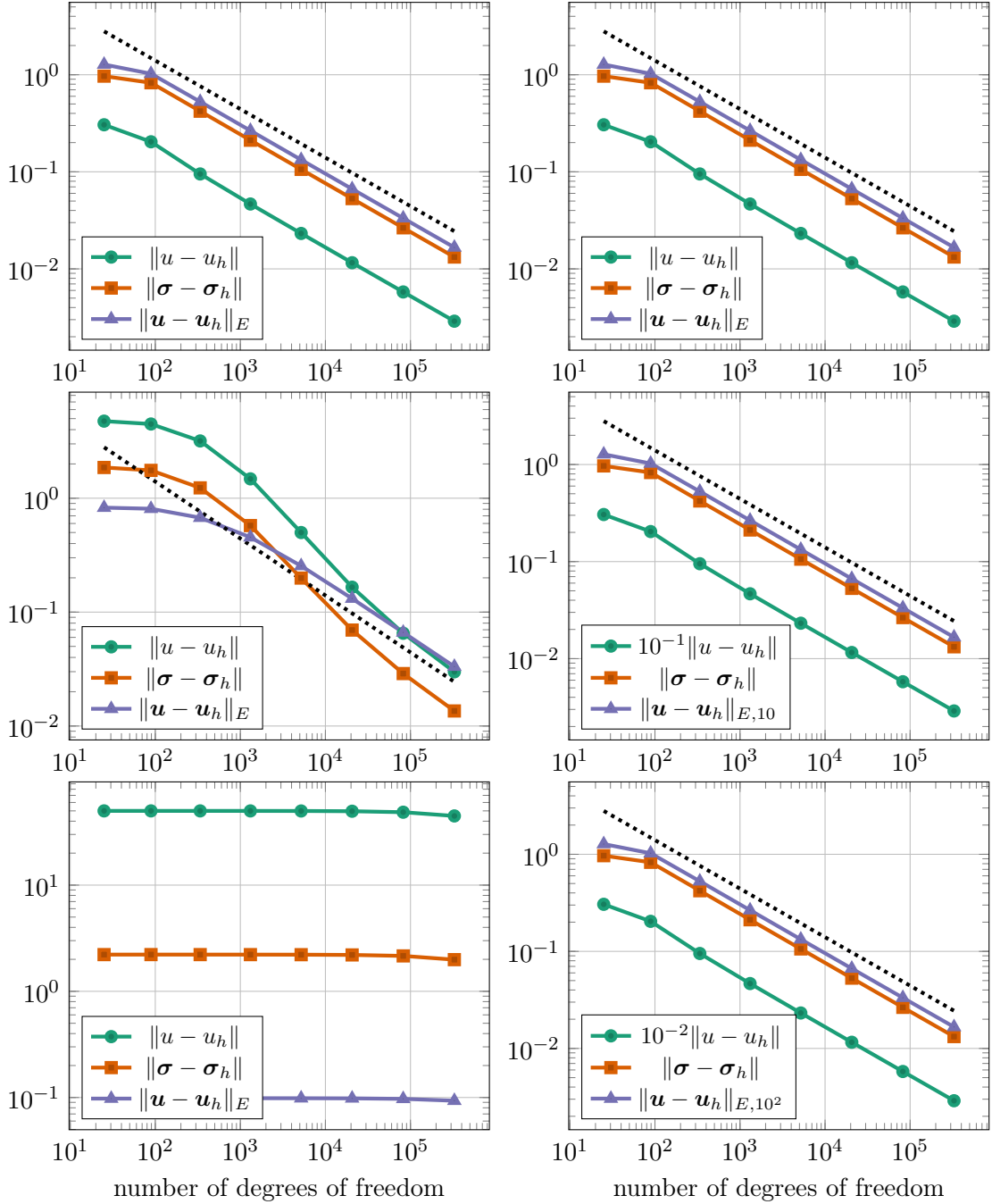
\begin{figure}
\begin{center}
\begin{tikzpicture}
\begin{loglogaxis}[
width=0.49\textwidth,
cycle list/Dark2-5,
cycle multiindex* list={
mark list*\nextlist
Dark2-5\nextlist
},
grid=major,
legend entries={\small $\|u-u_h\|$,\small $\|\bsigma-\bsigma_h\|$, \small $\|\bu-\bu_h\|_E$},
legend pos=south west,
every axis plot/.append style={ultra thick},
]
\addplot table [x=dofDPG,y=errU] {data/resultDiffusion_NoScaling_1.csv};
\addplot table [x=dofDPG,y=errSigma] {data/resultDiffusion_NoScaling_1.csv};
\addplot table [x=dofDPG,y=err] {data/resultDiffusion_NoScaling_1.csv};
\addplot [black,dotted,mark=none] table [x=dofDPG,y expr={14*sqrt(\thisrowno{1})^(-1)}] {data/resultDiffusion_NoScaling_1.csv};

\end{loglogaxis}
\end{tikzpicture}
\begin{tikzpicture}
\begin{loglogaxis}[
width=0.49\textwidth,
cycle list/Dark2-5,
cycle multiindex* list={
mark list*\nextlist
Dark2-5\nextlist
},
grid=major,
legend entries={\small $\|u-u_h\|$,\small $\|\bsigma-\bsigma_h\|$, \small $\|\bu-\bu_h\|_E$},
legend pos=south west,
every axis plot/.append style={ultra thick},
]
\addplot table [x=dofDPG,y=errU] {data/resultDiffusion_Scaling_1.csv};
\addplot table [x=dofDPG,y=errSigma] {data/resultDiffusion_Scaling_1.csv};
\addplot table [x=dofDPG,y=err] {data/resultDiffusion_Scaling_1.csv};
\addplot [black,dotted,mark=none] table [x=dofDPG,y expr={14*sqrt(\thisrowno{1})^(-1)}] {data/resultDiffusion_NoScaling_1.csv};

\end{loglogaxis}
\end{tikzpicture}
\begin{tikzpicture}
\begin{loglogaxis}[
width=0.49\textwidth,
cycle list/Dark2-5,
cycle multiindex* list={
mark list*\nextlist
Dark2-5\nextlist
},
grid=major,
legend entries={\small $\|u-u_h\|$,\small $\|\bsigma-\bsigma_h\|$, \small $\|\bu-\bu_h\|_E$},
legend pos=south west,
every axis plot/.append style={ultra thick},
]
\addplot table [x=dofDPG,y=errU] {data/resultDiffusion_NoScaling_2.csv};
\addplot table [x=dofDPG,y=errSigma] {data/resultDiffusion_NoScaling_2.csv};
\addplot table [x=dofDPG,y=err] {data/resultDiffusion_NoScaling_2.csv};
\addplot [black,dotted,mark=none] table [x=dofDPG,y expr={14*sqrt(\thisrowno{1})^(-1)}] {data/resultDiffusion_NoScaling_1.csv};

\end{loglogaxis}
\end{tikzpicture}
\begin{tikzpicture}
\begin{loglogaxis}[
width=0.49\textwidth,
cycle list/Dark2-5,
cycle multiindex* list={
mark list*\nextlist
Dark2-5\nextlist
},
grid=major,
legend entries={\small $10^{-1}\|u-u_h\|$,\small $\|\bsigma-\bsigma_h\|$, \small $\|\bu-\bu_h\|_{E,10}$},
legend pos=south west,
every axis plot/.append style={ultra thick},
]
\addplot table [x=dofDPG,y expr={\thisrowno{3}/10}] {data/resultDiffusion_Scaling_2.csv};
\addplot table [x=dofDPG,y=errSigma] {data/resultDiffusion_Scaling_2.csv};
\addplot table [x=dofDPG,y=err] {data/resultDiffusion_Scaling_2.csv};
\addplot [black,dotted,mark=none] table [x=dofDPG,y expr={14*sqrt(\thisrowno{1})^(-1)}] {data/resultDiffusion_NoScaling_1.csv};

\end{loglogaxis}
\end{tikzpicture}
\begin{tikzpicture}
\begin{loglogaxis}[
width=0.49\textwidth,
cycle list/Dark2-5,
cycle multiindex* list={
mark list*\nextlist
Dark2-5\nextlist
},
xlabel={number of degrees of freedom},
grid=major,
legend entries={\small $\|u-u_h\|$,\small $\|\bsigma-\bsigma_h\|$, \small $\|\bu-\bu_h\|_E$},
legend pos=south west,
every axis plot/.append style={ultra thick},
]
\addplot table [x=dofDPG,y=errU] {data/resultDiffusion_NoScaling_4.csv};
\addplot table [x=dofDPG,y=errSigma] {data/resultDiffusion_NoScaling_4.csv};
\addplot table [x=dofDPG,y=err] {data/resultDiffusion_NoScaling_4.csv};

\end{loglogaxis}
\end{tikzpicture}
\begin{tikzpicture}
\begin{loglogaxis}[
width=0.49\textwidth,
cycle list/Dark2-5,
cycle multiindex* list={
mark list*\nextlist
Dark2-5\nextlist
},
xlabel={number of degrees of freedom},
grid=major,
legend entries={\small $10^{-2}\|u-u_h\|$,\small $\|\bsigma-\bsigma_h\|$, \small $\|\bu-\bu_h\|_{E,10^2}$},
legend pos=south west,
every axis plot/.append style={ultra thick},
]
\addplot table [x=dofDPG,y expr={\thisrowno{3}/100}] {data/resultDiffusion_Scaling_4.csv};
\addplot table [x=dofDPG,y=errSigma] {data/resultDiffusion_Scaling_4.csv};
\addplot table [x=dofDPG,y=err] {data/resultDiffusion_Scaling_4.csv};
\addplot [black,dotted,mark=none] table [x=dofDPG,y expr={14*sqrt(\thisrowno{1})^(-1)}] {data/resultDiffusion_NoScaling_1.csv};

\end{loglogaxis}
\end{tikzpicture}
\end{center}
\caption{Results for the second-order problem from Section~\ref{sec:exp:secondorder} with $\gamma=0$
         on the computational domain $\Omega=(0,R)^2$ with $R=10^0,10^1,10^2$ (rows).
         The left and right column represents the results using the test norm
         $\|\cdot\|_V$ and $\|\cdot\|_{V,R}$, respectively.}
\label{fig:diffusion}
\end{figure}

For the second set of experiments we consider the same setting as before with $\gamma=1$.
The results are presented in Figure~\ref{fig:reaction}.
It is clear that the standard DPG scheme is robust, as expected by the stability analysis
(not reported here). Indeed, there is no need for a Poincar\'e inequality in this case.
For the same reason, when $\gamma=1$, there is no need or motivation to use scaled test norms.
Nevertheless, estimates still hold true but the scaling parameter $d$ introduces an
(in this case) unwanted imbalance between the errors in $u$ and $\bsigma$
(recall the factor $d^{-1}$ for $\|u-u_h\|$ which lowers this curve by a factor of
$10^{-2}$ in the last row of Figure~\ref{fig:reaction}).

\begin{figure}
\begin{center}
\begin{tikzpicture}
\begin{loglogaxis}[
width=0.49\textwidth,
cycle list/Dark2-5,
cycle multiindex* list={
mark list*\nextlist
Dark2-5\nextlist
},
grid=major,
legend entries={\small $\|u-u_h\|$,\small $\|\bsigma-\bsigma_h\|$, \small $\|\bu-\bu_h\|_E$},
legend pos=south west,
every axis plot/.append style={ultra thick},
]
\addplot table [x=dofDPG,y=errU] {data/resultReaction_NoScaling_1.csv};
\addplot table [x=dofDPG,y=errSigma] {data/resultReaction_NoScaling_1.csv};
\addplot table [x=dofDPG,y=err] {data/resultReaction_NoScaling_1.csv};
\addplot [black,dotted,mark=none] table [x=dofDPG,y expr={14*sqrt(\thisrowno{1})^(-1)}] {data/resultDiffusion_NoScaling_1.csv};

\end{loglogaxis}
\end{tikzpicture}
\begin{tikzpicture}
\begin{loglogaxis}[
width=0.49\textwidth,
cycle list/Dark2-5,
cycle multiindex* list={
mark list*\nextlist
Dark2-5\nextlist
},
grid=major,
legend entries={\small $\|u-u_h\|$,\small $\|\bsigma-\bsigma_h\|$, \small $\|\bu-\bu_h\|_E$},
legend pos=south west,
every axis plot/.append style={ultra thick},
]
\addplot table [x=dofDPG,y=errU] {data/resultReaction_Scaling_1.csv};
\addplot table [x=dofDPG,y=errSigma] {data/resultReaction_Scaling_1.csv};
\addplot table [x=dofDPG,y=err] {data/resultReaction_Scaling_1.csv};
\addplot [black,dotted,mark=none] table [x=dofDPG,y expr={14*sqrt(\thisrowno{1})^(-1)}] {data/resultDiffusion_NoScaling_1.csv};

\end{loglogaxis}
\end{tikzpicture}
\begin{tikzpicture}
\begin{loglogaxis}[
width=0.49\textwidth,
cycle list/Dark2-5,
cycle multiindex* list={
mark list*\nextlist
Dark2-5\nextlist
},
grid=major,
legend entries={\small $\|u-u_h\|$,\small $\|\bsigma-\bsigma_h\|$, \small $\|\bu-\bu_h\|_E$},
legend pos=south west,
every axis plot/.append style={ultra thick},
]
\addplot table [x=dofDPG,y=errU] {data/resultReaction_NoScaling_2.csv};
\addplot table [x=dofDPG,y=errSigma] {data/resultReaction_NoScaling_2.csv};
\addplot table [x=dofDPG,y=err] {data/resultReaction_NoScaling_2.csv};
\addplot [black,dotted,mark=none] table [x=dofDPG,y expr={14*sqrt(\thisrowno{1})^(-1)}] {data/resultDiffusion_NoScaling_1.csv};

\end{loglogaxis}
\end{tikzpicture}
\begin{tikzpicture}
\begin{loglogaxis}[
width=0.49\textwidth,
cycle list/Dark2-5,
cycle multiindex* list={
mark list*\nextlist
Dark2-5\nextlist
},
grid=major,
legend entries={\small $10^{-1}\|u-u_h\|$,\small $\|\bsigma-\bsigma_h\|$, \small $\|\bu-\bu_h\|_{E,10}$},
legend pos=south west,
every axis plot/.append style={ultra thick},
]
\addplot table [x=dofDPG,y expr={\thisrowno{3}/10}] {data/resultReaction_Scaling_2.csv};
\addplot table [x=dofDPG,y=errSigma] {data/resultReaction_Scaling_2.csv};
\addplot table [x=dofDPG,y=err] {data/resultReaction_Scaling_2.csv};
\addplot [black,dotted,mark=none] table [x=dofDPG,y expr={14*sqrt(\thisrowno{1})^(-1)}] {data/resultDiffusion_NoScaling_1.csv};

\end{loglogaxis}
\end{tikzpicture}
\begin{tikzpicture}
\begin{loglogaxis}[
width=0.49\textwidth,
cycle list/Dark2-5,
cycle multiindex* list={
mark list*\nextlist
Dark2-5\nextlist
},
xlabel={number of degrees of freedom},
grid=major,
legend entries={\small $\|u-u_h\|$,\small $\|\bsigma-\bsigma_h\|$, \small $\|\bu-\bu_h\|_E$},
legend pos=south west,
every axis plot/.append style={ultra thick},
]
\addplot table [x=dofDPG,y=errU] {data/resultReaction_NoScaling_4.csv};
\addplot table [x=dofDPG,y=errSigma] {data/resultReaction_NoScaling_4.csv};
\addplot table [x=dofDPG,y=err] {data/resultReaction_NoScaling_4.csv};
\addplot [black,dotted,mark=none] table [x=dofDPG,y expr={100*sqrt(\thisrowno{1})^(-1)}] {data/resultDiffusion_NoScaling_1.csv};

\end{loglogaxis}
\end{tikzpicture}
\begin{tikzpicture}
\begin{loglogaxis}[
width=0.49\textwidth,
cycle list/Dark2-5,
cycle multiindex* list={
mark list*\nextlist
Dark2-5\nextlist
},
xlabel={number of degrees of freedom},
grid=major,
legend entries={\small $10^{-2}\|u-u_h\|$,\small $\|\bsigma-\bsigma_h\|$, \small $\|\bu-\bu_h\|_{E,10^2}$},
legend pos=south west,
every axis plot/.append style={ultra thick},
]
\addplot table [x=dofDPG,y expr={\thisrowno{3}/100}] {data/resultReaction_Scaling_4.csv};
\addplot table [x=dofDPG,y=errSigma] {data/resultReaction_Scaling_4.csv};
\addplot table [x=dofDPG,y=err] {data/resultReaction_Scaling_4.csv};
\addplot [black,dotted,mark=none] table [x=dofDPG,y expr={4*sqrt(\thisrowno{1})^(-1)}] {data/resultDiffusion_NoScaling_1.csv};

\end{loglogaxis}
\end{tikzpicture}
\end{center}
\caption{Results for the second-order problem from Section~\ref{sec:exp:secondorder} with $\gamma=1$
         on the computational domain $\Omega=(0,R)^2$ with $R=10^0,10^1,10^2$ (rows).
         The left and right column represents the results using the test norm
         $\|\cdot\|_V$ and $\|\cdot\|_{V,R}$, respectively.}
\label{fig:reaction}
\end{figure}
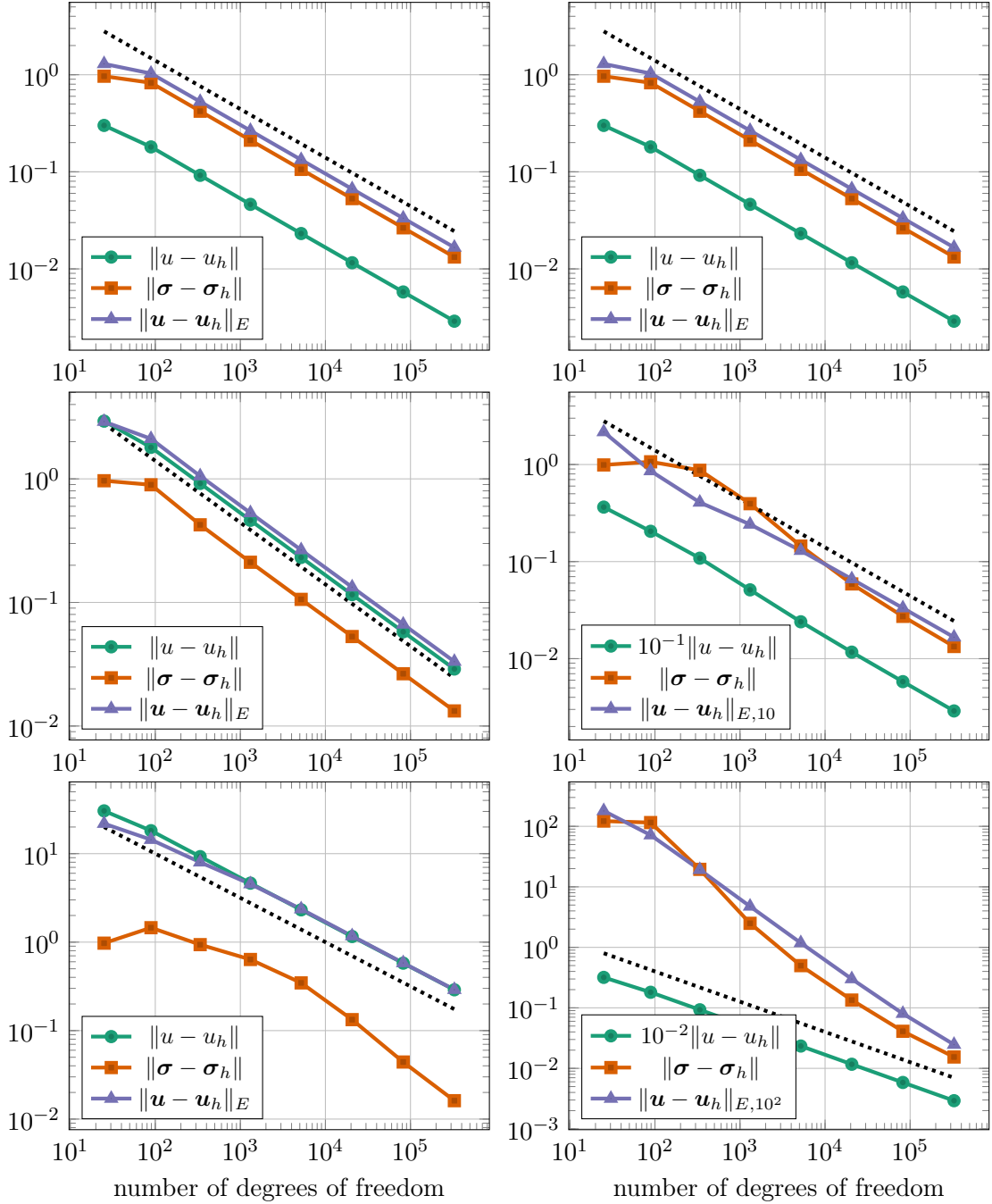

In the third experiment we consider $\gamma=0$ and the an-isotropic domain 
\begin{align*}
  \Omega = (0,R)\times (0,1)
\end{align*}
with $R\in\{10,100,1000\}$. Here, \eqref{PP} holds with $\cPP=1$ so that the standard
DPG method is identical to the scaled variant. Figure~\ref{fig:anisotropic} confirms that
there is no lack of robustness when using the standard DPG scheme. 

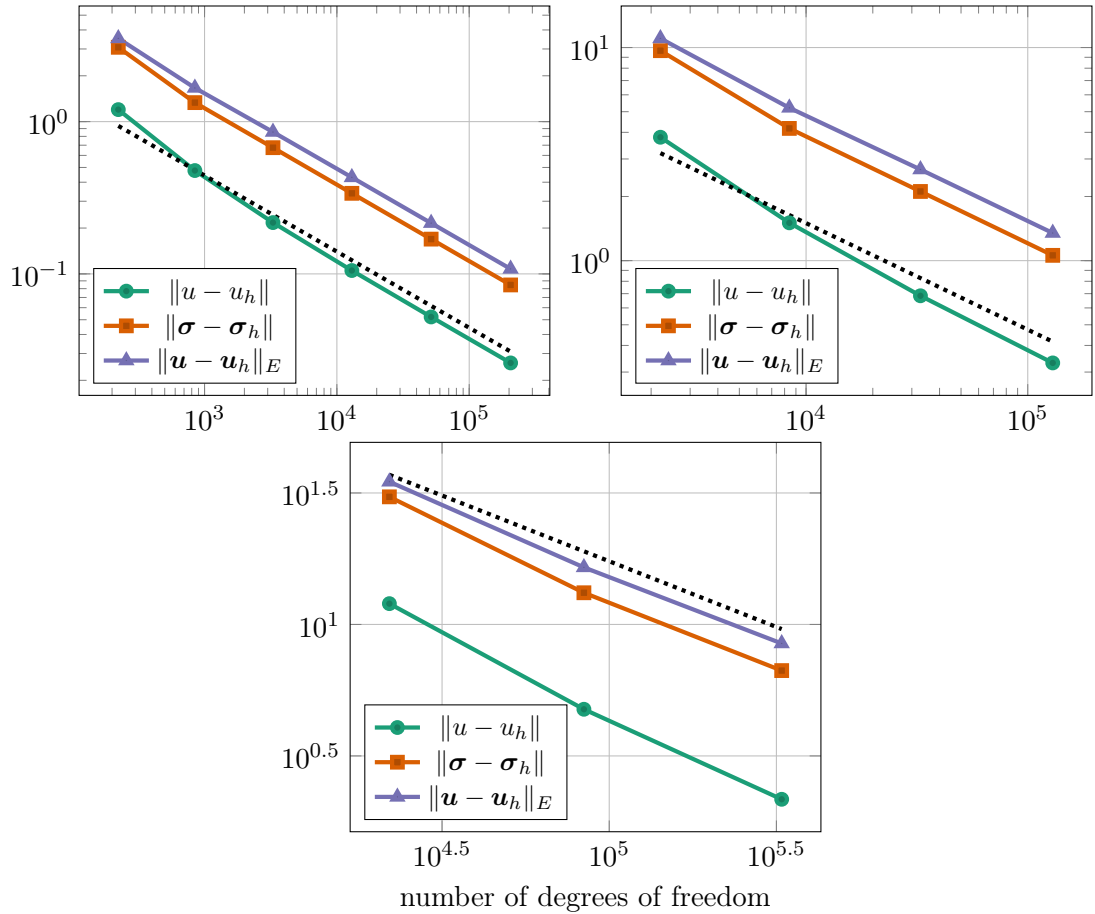
\begin{figure}
\begin{center}
\begin{tikzpicture}
\begin{loglogaxis}[
width=0.49\textwidth,
cycle list/Dark2-5,
cycle multiindex* list={
mark list*\nextlist
Dark2-5\nextlist
},
grid=major,
legend entries={\small $\|u-u_h\|$,\small $\|\bsigma-\bsigma_h\|$, \small $\|\bu-\bu_h\|_E$},
legend pos=south west,
every axis plot/.append style={ultra thick},
]
\addplot table [x=dofDPG,y=errU] {data/resultDiffusionAnisotropic_NoScaling_2.csv};
\addplot table [x=dofDPG,y=errSigma] {data/resultDiffusionAnisotropic_NoScaling_2.csv};
\addplot table [x=dofDPG,y=err] {data/resultDiffusionAnisotropic_NoScaling_2.csv};
\addplot [black,dotted,mark=none] table [x=dofDPG,y expr={14*sqrt(\thisrowno{1})^(-1)}] {data/resultDiffusionAnisotropic_NoScaling_2.csv};

\end{loglogaxis}
\end{tikzpicture}
%
\begin{tikzpicture}
\begin{loglogaxis}[
width=0.49\textwidth,
cycle list/Dark2-5,
cycle multiindex* list={
mark list*\nextlist
Dark2-5\nextlist
},
grid=major,
legend entries={\small $\|u-u_h\|$,\small $\|\bsigma-\bsigma_h\|$, \small $\|\bu-\bu_h\|_E$},
legend pos=south west,
every axis plot/.append style={ultra thick},
]
\addplot table [x=dofDPG,y=errU] {data/resultDiffusionAnisotropic_NoScaling_4.csv};
\addplot table [x=dofDPG,y=errSigma] {data/resultDiffusionAnisotropic_NoScaling_4.csv};
\addplot table [x=dofDPG,y=err] {data/resultDiffusionAnisotropic_NoScaling_4.csv};
\addplot [black,dotted,mark=none] table [x=dofDPG,y expr={150*sqrt(\thisrowno{1})^(-1)}] {data/resultDiffusionAnisotropic_NoScaling_4.csv};

\end{loglogaxis}
\end{tikzpicture}
%
\begin{tikzpicture}
\begin{loglogaxis}[
width=0.49\textwidth,
cycle list/Dark2-5,
cycle multiindex* list={
mark list*\nextlist
Dark2-5\nextlist
},
xlabel={number of degrees of freedom},
grid=major,
legend entries={\small $\|u-u_h\|$,\small $\|\bsigma-\bsigma_h\|$, \small $\|\bu-\bu_h\|_E$},
legend pos=south west,
every axis plot/.append style={ultra thick},
]
\addplot table [x=dofDPG,y=errU] {data/resultDiffusionAnisotropic_NoScaling_5.csv};
\addplot table [x=dofDPG,y=errSigma] {data/resultDiffusionAnisotropic_NoScaling_5.csv};
\addplot table [x=dofDPG,y=err] {data/resultDiffusionAnisotropic_NoScaling_5.csv};
\addplot [black,dotted,mark=none] table [x=dofDPG,y expr={5500*sqrt(\thisrowno{1})^(-1)}] {data/resultDiffusionAnisotropic_NoScaling_5.csv};

\end{loglogaxis}
\end{tikzpicture}
%
\end{center}
\caption{Results for the second-order problem from Section~\ref{sec:exp:secondorder} with $\gamma=0$
         on the computational domain $\Omega=(0,R)\times(0,1)$ with $R=10^1,10^2,10^3$,
       using the standard DPG scheme.}
\label{fig:anisotropic}
\end{figure}

\subsection{Poisson problem with mixed boundary conditions}\label{sec:exp:mixedBC}
We consider the an-isotropic domain $\Omega = (0,R)\times (0,1)$ with boundary parts
$\Gamma_D = \{0\}\times [0,1] \cup \{1\}\times[0,1]$ and
$\Gamma_N = \partial\Omega\setminus\overline\Gamma_D$.
Our manufactured solution $u(x,y) = \sin(\pi x/R)$ satisfies the Poisson problem
with mixed boundary conditions,
\begin{align*}
  -\Delta u &= \frac{\pi^2}{R^2} \sin(\pi x/R)\quad\text{in}\ \Omega,\quad
          u|_{\Gamma_D} = 0,\quad
          \partial_{\bn}u|_{\Gamma_N} = 0.
\end{align*}
In this case the Poincar\'e inequality \eqref{PP} is satisfied with
$d=d(\Omega)=R$, cf.~\cite[II.1.5]{Braess_97_FET}.
From our analysis we expect that the standard DPG scheme is not robust with respect to the domain size.
This is observed in Figure~\ref{fig:anisotropic:mixedBC} (left column). Again, control
of the field variables by the energy norm is lost.
One sees that the use of scaled test norms with $d=R$ establishes robustness (right column).

\begin{figure}
\begin{center}
\begin{tikzpicture}
\begin{loglogaxis}[
width=0.49\textwidth,
cycle list/Dark2-5,
cycle multiindex* list={
mark list*\nextlist
Dark2-5\nextlist
},
grid=major,
legend entries={\small $\|u-u_h\|$,\small $\|\bsigma-\bsigma_h\|$, \small $\|\bu-\bu_h\|_E$},
legend pos=south west,
every axis plot/.append style={ultra thick},
]
\addplot table [x=dofDPG,y=errU] {data/resultDiffusionAnisotropicMixedBC_NoScaling_1.csv};
\addplot table [x=dofDPG,y=errSigma] {data/resultDiffusionAnisotropicMixedBC_NoScaling_1.csv};
\addplot table [x=dofDPG,y=err] {data/resultDiffusionAnisotropicMixedBC_NoScaling_1.csv};
\addplot [black,dotted,mark=none] table [x=dofDPG,y expr={4*sqrt(\thisrowno{1})^(-1)}] {data/resultDiffusionAnisotropicMixedBC_NoScaling_1.csv};

\end{loglogaxis}
\end{tikzpicture}
\begin{tikzpicture}
\begin{loglogaxis}[
width=0.49\textwidth,
cycle list/Dark2-5,
cycle multiindex* list={
mark list*\nextlist
Dark2-5\nextlist
},
grid=major,
legend entries={\small $10^{-1}\|u-u_h\|$,\small $\|\bsigma-\bsigma_h\|$, \small $\|\bu-\bu_h\|_{E,10}$},
legend pos=south west,
every axis plot/.append style={ultra thick},
]
\addplot table [x=dofDPG,y expr={\thisrowno{3}/10}] {data/resultDiffusionAnisotropicMixedBC_Scaling_1.csv};
\addplot table [x=dofDPG,y=errSigma] {data/resultDiffusionAnisotropicMixedBC_Scaling_1.csv};
\addplot table [x=dofDPG,y=err] {data/resultDiffusionAnisotropicMixedBC_Scaling_1.csv};
\addplot [black,dotted,mark=none] table [x=dofDPG,y expr={sqrt(\thisrowno{1})^(-1)}] {data/resultDiffusionAnisotropicMixedBC_Scaling_1.csv};

\end{loglogaxis}
\end{tikzpicture}
\begin{tikzpicture}
\begin{loglogaxis}[
width=0.49\textwidth,
cycle list/Dark2-5,
cycle multiindex* list={
mark list*\nextlist
Dark2-5\nextlist
},
grid=major,
legend entries={\small $\|u-u_h\|$,\small $\|\bsigma-\bsigma_h\|$, \small $\|\bu-\bu_h\|_E$},
legend pos=south west,
every axis plot/.append style={ultra thick},
]
\addplot table [x=dofDPG,y=errU] {data/resultDiffusionAnisotropicMixedBC_NoScaling_2.csv};
\addplot table [x=dofDPG,y=errSigma] {data/resultDiffusionAnisotropicMixedBC_NoScaling_2.csv};
\addplot table [x=dofDPG,y=err] {data/resultDiffusionAnisotropicMixedBC_NoScaling_2.csv};
\addplot [black,dotted,mark=none] table [x=dofDPG,y expr={14*sqrt(\thisrowno{1})^(-1)}] {data/resultDiffusionAnisotropicMixedBC_NoScaling_2.csv};

\end{loglogaxis}
\end{tikzpicture}
\begin{tikzpicture}
\begin{loglogaxis}[
width=0.49\textwidth,
cycle list/Dark2-5,
cycle multiindex* list={
mark list*\nextlist
Dark2-5\nextlist
},
grid=major,
legend entries={\small $10^{-2}\|u-u_h\|$,\small $\|\bsigma-\bsigma_h\|$, \small $\|\bu-\bu_h\|_{E,10^2}$},
legend pos=south west,
every axis plot/.append style={ultra thick},
]
\addplot table [x=dofDPG,y expr={\thisrowno{3}/100}] {data/resultDiffusionAnisotropicMixedBC_Scaling_2.csv};
\addplot table [x=dofDPG,y=errSigma] {data/resultDiffusionAnisotropicMixedBC_Scaling_2.csv};
\addplot table [x=dofDPG,y=err] {data/resultDiffusionAnisotropicMixedBC_Scaling_2.csv};
\addplot [black,dotted,mark=none] table [x=dofDPG,y expr={1/10*sqrt(\thisrowno{1})^(-1)}] {data/resultDiffusionAnisotropicMixedBC_Scaling_2.csv};

\end{loglogaxis}
\end{tikzpicture}
\begin{tikzpicture}
\begin{loglogaxis}[
width=0.49\textwidth,
cycle list/Dark2-5,
cycle multiindex* list={
mark list*\nextlist
Dark2-5\nextlist
},
xlabel={number of degrees of freedom},
grid=major,
legend entries={\small $\|u-u_h\|$,\small $\|\bsigma-\bsigma_h\|$, \small $\|\bu-\bu_h\|_E$},
legend pos=south west,
every axis plot/.append style={ultra thick},
]
\addplot table [x=dofDPG,y=errU] {data/resultDiffusionAnisotropicMixedBC_NoScaling_3.csv};
\addplot table [x=dofDPG,y=errSigma] {data/resultDiffusionAnisotropicMixedBC_NoScaling_3.csv};
\addplot table [x=dofDPG,y=err] {data/resultDiffusionAnisotropicMixedBC_NoScaling_3.csv};

\end{loglogaxis}
\end{tikzpicture}
\begin{tikzpicture}
\begin{loglogaxis}[
width=0.49\textwidth,
cycle list/Dark2-5,
cycle multiindex* list={
mark list*\nextlist
Dark2-5\nextlist
},
xlabel={number of degrees of freedom},
grid=major,
legend entries={\small $10^{-3}\|u-u_h\|$,\small $\|\bsigma-\bsigma_h\|$, \small $\|\bu-\bu_h\|_{E,10^3}$},
legend pos=south west,
every axis plot/.append style={ultra thick},
]
\addplot table [x=dofDPG,y expr={\thisrowno{3}/1000}] {data/resultDiffusionAnisotropicMixedBC_Scaling_3.csv};
\addplot table [x=dofDPG,y=errSigma] {data/resultDiffusionAnisotropicMixedBC_Scaling_3.csv};
\addplot table [x=dofDPG,y=err] {data/resultDiffusionAnisotropicMixedBC_Scaling_3.csv};
\addplot [black,dotted,mark=none] table [x=dofDPG,y expr={1/100*sqrt(\thisrowno{1})^(-1)}] {data/resultDiffusionAnisotropicMixedBC_Scaling_3.csv};

\end{loglogaxis}
\end{tikzpicture}
\end{center}
\caption{Results for the second-order problem from Section~\ref{sec:exp:mixedBC}
  on the computational domain $\Omega=(0,R)\times(0,1)$ with $R=10^1,10^2,10^3$ (rows),
  using the standard DPG scheme (left column) and the DPG scheme with scaled test norm (right column).}
\label{fig:anisotropic:mixedBC}
\end{figure}
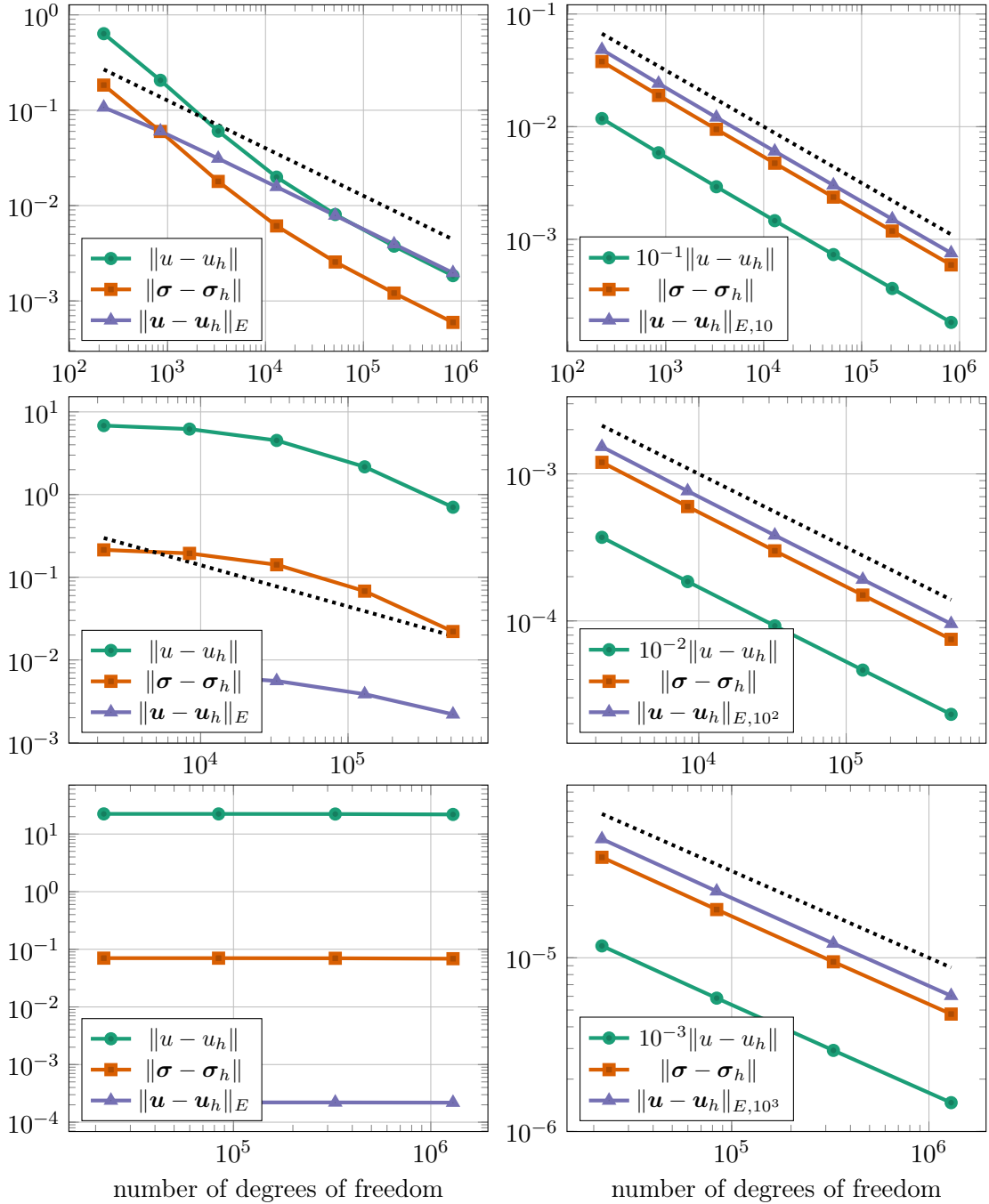

\subsection{Kirchhoff--Love model}\label{sec:exp:fourthorder}

We consider the manufactured solution 
\begin{align*}
  u(x,y) = \sin(\pi x/R_1)^2\sin(\pi y/R_2)^2
\end{align*}
on the computational domain $\Omega = (0,R_1)\times(0,R_2)$, use the identity for tensor
$\cC$ (the problem then reduces to the bi-Laplacian), specify the boundary condition of a
clamped plate, and define the right-hand side function $f:=\Delta^2 u$.
In this case the Poincar\'e-type estimate \eqref{PKL} holds with $\cPP=\min\{R_1,R_2\}$,
cf.~\cite[II.1.5,1.6]{Braess_97_FET}, which is our choice for parameter $d$.

In the displayed results we use the notation
\begin{align*}
  \|\bu-\bu_h\|_{E,d} &:= \sup_{0\neq \bv\in V_h} \frac{b(\bu-\bu_h,\bv)}{\|\bv\|_{V,d}},
  \quad
  \|\bu-\bu_h\|_{E} := \|\bu-\bu_h\|_{E,1}.
\end{align*}
For the first set of experiments we choose $R_1=R_2=R$. 
Figure~\ref{fig:KLove} visualizes the results for domains with $R=1,10,100$ (rows). 
In each case, the errors for $u$ and $\MM=-\Grad\grad u$ along with the energy errors are shown.
Again, the left and right columns correspond to DPG methods using
the test norms $\|\cdot\|_V$ and $\|\cdot\|_{V,d}$, respectively. 
As in the case of the second-order elliptic problem, we observe locking of the standard
DPG scheme whereas, when switching to the scaled test norm $\|\cdot\|_{V,d}$, robustness
is re-established.

Numerical experiments with an-isotropic domains $\Omega=(0,R)\times (0,1)$,
and clamped condition on the whole of the boundary, confirm that there is no need for
scaled test norms to guarantee robustness, as we have observed for the Poisson problem.

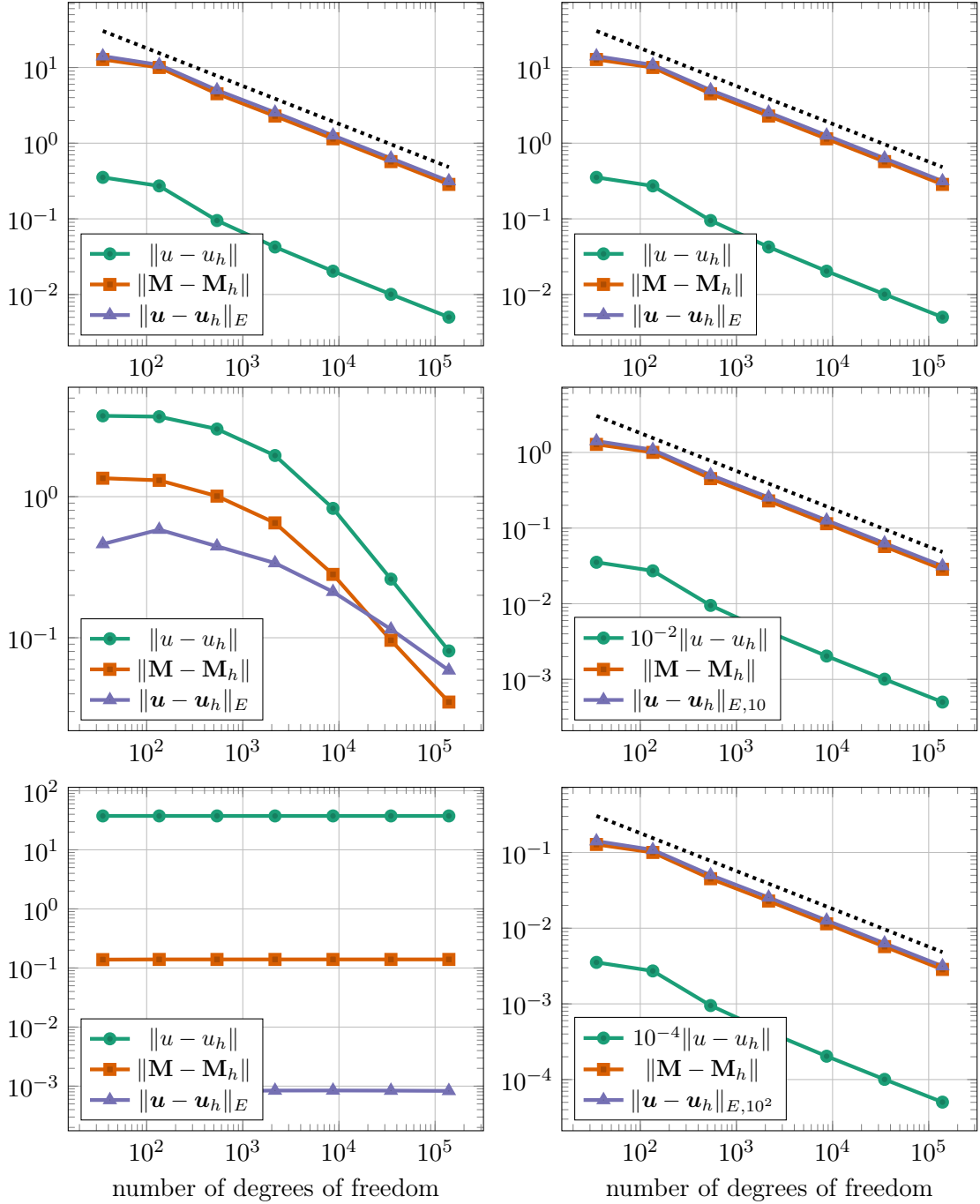
\begin{figure}
\begin{center}
\begin{tikzpicture}
\begin{loglogaxis}[
width=0.49\textwidth,
cycle list/Dark2-5,
cycle multiindex* list={
mark list*\nextlist
Dark2-5\nextlist
},
grid=major,
legend entries={\small $\|u-u_h\|$,\small $\|\MM-\MM_h\|$, \small $\|\bu-\bu_h\|_E$},
legend pos=south west,
every axis plot/.append style={ultra thick},
]
\addplot table [x=dofDPG,y=errU] {data/resultKLove_NoScaling_1.csv};
\addplot table [x=dofDPG,y=errSigma] {data/resultKLove_NoScaling_1.csv};
\addplot table [x=dofDPG,y=err] {data/resultKLove_NoScaling_1.csv};
\addplot [black,dotted,mark=none] table [x=dofDPG,y expr={18*10*sqrt(\thisrowno{1})^(-1)}] {data/resultKLove_NoScaling_1.csv};

\end{loglogaxis}
\end{tikzpicture}
\begin{tikzpicture}
\begin{loglogaxis}[
width=0.49\textwidth,
cycle list/Dark2-5,
cycle multiindex* list={
mark list*\nextlist
Dark2-5\nextlist
},
grid=major,
legend entries={\small $\|u-u_h\|$,\small $\|\MM-\MM_h\|$, \small $\|\bu-\bu_h\|_E$},
legend pos=south west,
every axis plot/.append style={ultra thick},
]
\addplot table [x=dofDPG,y=errU] {data/resultKLove_Scaling_1.csv};
\addplot table [x=dofDPG,y=errSigma] {data/resultKLove_Scaling_1.csv};
\addplot table [x=dofDPG,y=err] {data/resultKLove_Scaling_1.csv};
\addplot [black,dotted,mark=none] table [x=dofDPG,y expr={18*10*sqrt(\thisrowno{1})^(-1)}] {data/resultKLove_NoScaling_1.csv};

\end{loglogaxis}
\end{tikzpicture}
\begin{tikzpicture}
\begin{loglogaxis}[
width=0.49\textwidth,
cycle list/Dark2-5,
cycle multiindex* list={
mark list*\nextlist
Dark2-5\nextlist
},
grid=major,
legend entries={\small $\|u-u_h\|$,\small $\|\MM-\MM_h\|$, \small $\|\bu-\bu_h\|_E$},
legend pos=south west,
every axis plot/.append style={ultra thick},
]
\addplot table [x=dofDPG,y=errU] {data/resultKLove_NoScaling_2.csv};
\addplot table [x=dofDPG,y=errSigma] {data/resultKLove_NoScaling_2.csv};
\addplot table [x=dofDPG,y=err] {data/resultKLove_NoScaling_2.csv};

\end{loglogaxis}
\end{tikzpicture}
\begin{tikzpicture}
\begin{loglogaxis}[
width=0.49\textwidth,
cycle list/Dark2-5,
cycle multiindex* list={
mark list*\nextlist
Dark2-5\nextlist
},
grid=major,
legend entries={\small $10^{-2}\|u-u_h\|$,\small $\|\MM-\MM_h\|$, \small $\|\bu-\bu_h\|_{E,10}$},
legend pos=south west,
every axis plot/.append style={ultra thick},
]
\addplot table [x=dofDPG,y expr={\thisrowno{3}/100}] {data/resultKLove_Scaling_2.csv};
\addplot table [x=dofDPG,y=errSigma] {data/resultKLove_Scaling_2.csv};
\addplot table [x=dofDPG,y=err] {data/resultKLove_Scaling_2.csv};
\addplot [black,dotted,mark=none] table [x=dofDPG,y expr={18*sqrt(\thisrowno{1})^(-1)}] {data/resultKLove_NoScaling_1.csv};

\end{loglogaxis}
\end{tikzpicture}
\begin{tikzpicture}
\begin{loglogaxis}[
width=0.49\textwidth,
cycle list/Dark2-5,
cycle multiindex* list={
mark list*\nextlist
Dark2-5\nextlist
},
xlabel={number of degrees of freedom},
grid=major,
legend entries={\small $\|u-u_h\|$,\small $\|\MM-\MM_h\|$, \small $\|\bu-\bu_h\|_E$},
legend pos=south west,
every axis plot/.append style={ultra thick},
]
\addplot table [x=dofDPG,y=errU] {data/resultKLove_NoScaling_3.csv};
\addplot table [x=dofDPG,y=errSigma] {data/resultKLove_NoScaling_3.csv};
\addplot table [x=dofDPG,y=err] {data/resultKLove_NoScaling_3.csv};

\end{loglogaxis}
\end{tikzpicture}
\begin{tikzpicture}
\begin{loglogaxis}[
width=0.49\textwidth,
cycle list/Dark2-5,
cycle multiindex* list={
mark list*\nextlist
Dark2-5\nextlist
},
xlabel={number of degrees of freedom},
grid=major,
legend entries={\small $10^{-4}\|u-u_h\|$,\small $\|\MM-\MM_h\|$, \small $\|\bu-\bu_h\|_{E,10^2}$},
legend pos=south west,
every axis plot/.append style={ultra thick},
]
\addplot table [x=dofDPG,y expr={\thisrowno{3}/10000}] {data/resultKLove_Scaling_3.csv};
\addplot table [x=dofDPG,y=errSigma] {data/resultKLove_Scaling_3.csv};
\addplot table [x=dofDPG,y=err] {data/resultKLove_Scaling_3.csv};
\addplot [black,dotted,mark=none] table [x=dofDPG,y expr={18/10*sqrt(\thisrowno{1})^(-1)}] {data/resultKLove_NoScaling_1.csv};

\end{loglogaxis}
\end{tikzpicture}
\end{center}
\caption{Results for the fourth-order problem from Section~\ref{sec:exp:fourthorder}
         on the computational domain $\Omega=(0,R)^2$ with $R=10^0,10^1,10^2$ (rows).
         The left and right column represents the results using the test norm
         $\|\cdot\|_V$ and $\|\cdot\|_{V,d}$, respectively.}
\label{fig:KLove}
\end{figure}

\subsection{Kirchhoff--Love model with mixed boundary conditions}\label{sec:exp:KLove:mixedBC}
We consider the an-isotropic domain $\Omega = (0,R)\times (0,1)$ with boundary parts
$\Gamma_D = \{0\}\times [0,1] \cup \{1\}\times[0,1]$ and
$\Gamma_N = \partial\Omega\setminus\overline\Gamma_D$.
We define $f:=\Delta^2 u$ with manufactured solution $u(x,y) = \sin(\pi x/R)^2$,
which satisfies the Kirchhoff--Love plate bending problem with mixed boundary conditions,
\begin{align*}
  -\div\Div \MM &= f\quad\text{in}\ \Omega, \\
  \MM + \Grad\grad u &= 0\quad\text{in}\ \Omega, \\
  u = \partial_{\bn} u &= 0\quad\text{on}\ \Gamma_D
                                           \qquad\text{(clamped)}, \\
  \bn\cdot\MM\bn = \bn\cdot\Div\MM + \partial_{\bt}(\bt\cdot\MM\bn) &= 0\quad\text{on}\ \Gamma_N
                                           \qquad\text{(free)}, \\
  [\bt\cdot\MM\bn](z) &= 0 \quad \text{for nodes } z\in \Gamma_N.
\end{align*}
Here, $[\cdot](z)$ denotes the jump along the boundary in $z$,
and $\bt$ is the tangential unit vector along the boundary in mathematically positive orientation.
(The normal vector $\bn$ and tangential derivative $\partial_\bt$ have been defined previously.)
In this case, i.e., $v\in H^2(\Omega)$ with $v=\partial_\bn v=0$ on $\Gamma_D$,
the Poincar\'e inequality \eqref{PKL} holds with $\cPP=R$, cf.~\cite[II.1.5,1.6]{Braess_97_FET}.

From our analysis we expect that the standard DPG scheme fails to be robust with respect
to the domain size. This is confirmed in Figure~\ref{fig:KLove:mixedBC} (left column). 
In comparison, we observe that our DPG scheme with scaled test norm ($d=R$)
is locking free (right column).

\begin{figure}
\begin{center}
\begin{tikzpicture}
\begin{loglogaxis}[
width=0.49\textwidth,
cycle list/Dark2-5,
cycle multiindex* list={
mark list*\nextlist
Dark2-5\nextlist
},
grid=major,
legend entries={\small $\|u-u_h\|$,\small $\|\MM-\MM_h\|$, \small $\|\bu-\bu_h\|_E$},
legend pos=south west,
every axis plot/.append style={ultra thick},
]
\addplot table [x=dofDPG,y=errU] {data/resultKLoveMixedBC_NoScaling_1.csv};
\addplot table [x=dofDPG,y=errSigma] {data/resultKLoveMixedBC_NoScaling_1.csv};
\addplot table [x=dofDPG,y=err] {data/resultKLoveMixedBC_NoScaling_1.csv};
\addplot [black,dotted,mark=none] table [x=dofDPG,y expr={4*sqrt(\thisrowno{1})^(-1)}] {data/resultKLoveMixedBC_NoScaling_1.csv};

\end{loglogaxis}
\end{tikzpicture}
\begin{tikzpicture}
\begin{loglogaxis}[
width=0.49\textwidth,
cycle list/Dark2-5,
cycle multiindex* list={
mark list*\nextlist
Dark2-5\nextlist
},
grid=major,
legend entries={\small $10^{-2}\|u-u_h\|$,\small $\|\MM-\MM_h\|$, \small $\|\bu-\bu_h\|_{E,10}$},
legend pos=south west,
every axis plot/.append style={ultra thick},
]
\addplot table [x=dofDPG,y expr={\thisrowno{3}/100}] {data/resultKLoveMixedBC_Scaling_1.csv};
\addplot table [x=dofDPG,y=errSigma] {data/resultKLoveMixedBC_Scaling_1.csv};
\addplot table [x=dofDPG,y=err] {data/resultKLoveMixedBC_Scaling_1.csv};
\addplot [black,dotted,mark=none] table [x=dofDPG,y expr={1.5*sqrt(\thisrowno{1})^(-1)}] {data/resultKLoveMixedBC_Scaling_1.csv};

\end{loglogaxis}
\end{tikzpicture}
\begin{tikzpicture}
\begin{loglogaxis}[
width=0.49\textwidth,
cycle list/Dark2-5,
cycle multiindex* list={
mark list*\nextlist
Dark2-5\nextlist
},
xlabel={number of degrees of freedom},
grid=major,
legend entries={\small $\|u-u_h\|$,\small $\|\MM-\MM_h\|$, \small $\|\bu-\bu_h\|_E$},
legend pos=south west,
every axis plot/.append style={ultra thick},
]
\addplot table [x=dofDPG,y=errU] {data/resultKLoveMixedBC_NoScaling_2.csv};
\addplot table [x=dofDPG,y=errSigma] {data/resultKLoveMixedBC_NoScaling_2.csv};
\addplot table [x=dofDPG,y=err] {data/resultKLoveMixedBC_NoScaling_2.csv};

\end{loglogaxis}
\end{tikzpicture}
\begin{tikzpicture}
\begin{loglogaxis}[
width=0.49\textwidth,
cycle list/Dark2-5,
cycle multiindex* list={
mark list*\nextlist
Dark2-5\nextlist
},
xlabel={number of degrees of freedom},
grid=major,
legend entries={\small $25^{-2}\|u-u_h\|$,\small $\|\MM-\MM_h\|$, \small $\|\bu-\bu_h\|_{E,25}$},
legend style={at={(0.03,0.5)},anchor=west},
every axis plot/.append style={ultra thick},
]
\addplot table [x=dofDPG,y expr={\thisrowno{3}/(25*25)}] {data/resultKLoveMixedBC_Scaling_2.csv};
\addplot table [x=dofDPG,y=errSigma] {data/resultKLoveMixedBC_Scaling_2.csv};
\addplot table [x=dofDPG,y=err] {data/resultKLoveMixedBC_Scaling_2.csv};
\addplot [black,dotted,mark=none] table [x=dofDPG,y expr={1/10*sqrt(\thisrowno{1})^(-1)}] {data/resultKLoveMixedBC_Scaling_2.csv};

\end{loglogaxis}
\end{tikzpicture}
\end{center}
\caption{Results for the fourth-order problem from Section~\ref{sec:exp:KLove:mixedBC}
  on the computational domain $\Omega=(0,R)\times(0,1)$ with $R=10,25$ (rows),
  using the standard DPG scheme (left column) and the DPG scheme with scaled test norm (right column).}
\label{fig:KLove:mixedBC}
\end{figure}
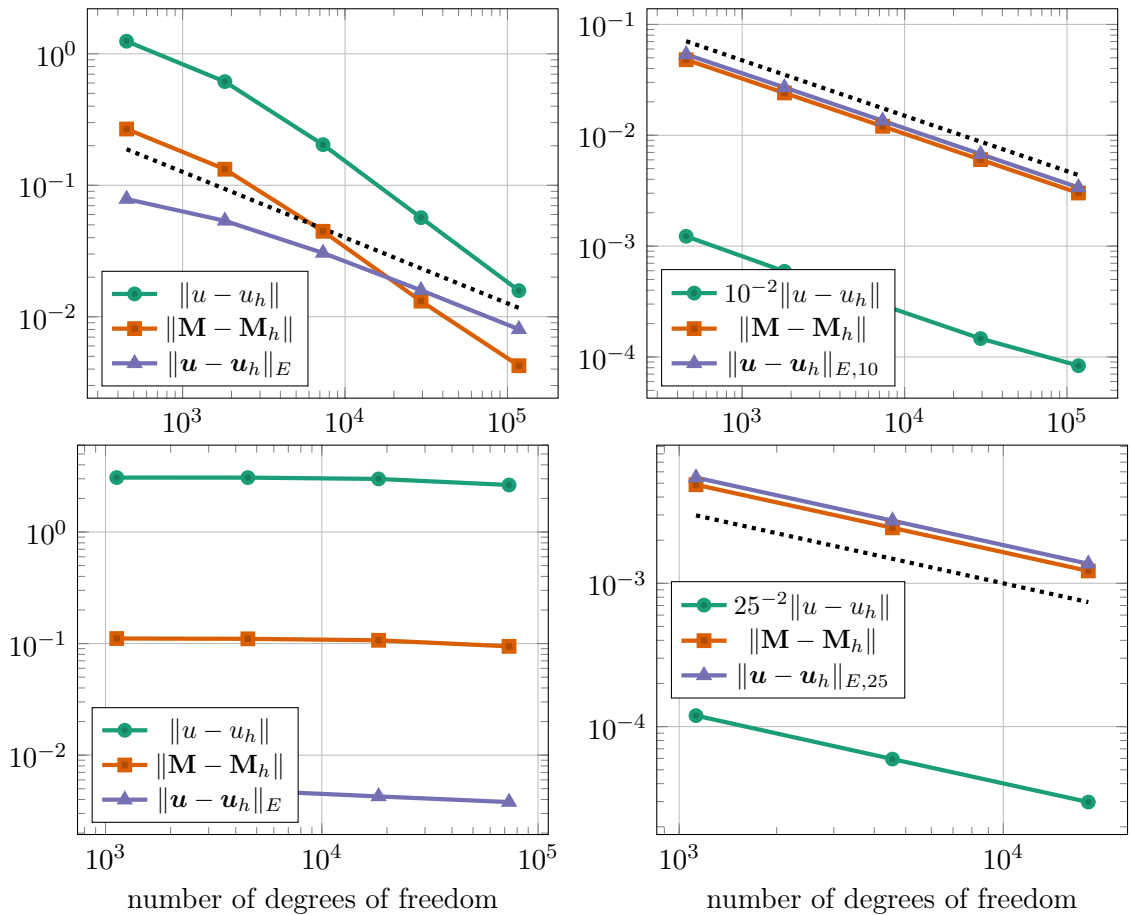

%

\section*{Conclusions}

We have seen that the DPG method exhibits domain locking under four conditions:
\begin{enumerate}
\item an ultraweak formulation is used,
\item standard test norms are taken to calculate optimal test functions,
\item stability of the adjoint problem hinges on a Poincar\'e-type estimate, and
\item boundary conditions are such that the Poincar\'e inequality lacks uniformity
      with respect to the (sequence of) domain(s) under consideration.
\end{enumerate}
This locking means that there is a pre-asymptotic range of reduced convergence
(or no error reduction at all) and a loss of control of the field variables by the energy norm.
Higher-order problems are substantially more susceptible to domain locking than lower-order
problems.
This locking phenomenon can be avoided by switching to test norms where individual
terms are appropriately scaled, as indicated by the Poincar\'e ``constant'' of the problem.


\clearpage
\bibliographystyle{abbrv}
\bibliography{/home/norbert/tex/bib/bib,/home/norbert/tex/bib/heuer,/home/norbert/tex/bib/fem}
\end{document}